\documentclass[11pt]{amsart}
\usepackage{graphicx}
\usepackage{amsmath}
\usepackage{color}
\usepackage{amsfonts}
\usepackage{amssymb}
\usepackage[utf8]{inputenc}
\usepackage[T1]{fontenc}
\usepackage{amsfonts,enumerate}
\usepackage{amsmath}
\usepackage{amssymb}
\usepackage{amsthm}
\usepackage{graphics}
\usepackage[backref = page]{hyperref}
\usepackage{tikz}
\usepackage{geometry}
\usepackage{amsfonts}
\usepackage{graphicx}
\usepackage{genyoungtabtikz}
\usepackage{dsfont}
\usepackage{mathrsfs}
\usepackage{ytableau}
\usepackage{tcolorbox}
\usepackage{color}
\usepackage{makecell}
\usepackage{longtable}
\usepackage[capitalize]{cleveref}
\usepackage{dynkin-diagrams}
\usepackage{float}%
\providecommand{\U}[1]{\protect\rule{.1in}{.1in}}

\hypersetup{
	colorlinks =  true,
	linkcolor = purple!80,
	citecolor = teal
}
\newcommand{\la}{\lambda}

\allowdisplaybreaks

\providecommand{\U}[1]{\protect\rule{.1in}{.1in}}
\providecommand{\U}[1]{\protect\rule{.1in}{.1in}}
\providecommand{\U}[1]{\protect\rule{.1in}{.1in}}
\providecommand{\U}[1]{\protect\rule{.1in}{.1in}}
\providecommand{\U}[1]{\protect\rule{.1in}{.1in}}
\providecommand{\U}[1]{\protect\rule{.1in}{.1in}}
\providecommand{\U}[1]{\protect\rule{.1in}{.1in}}
\theoremstyle{plain}
\newtheorem{theorem}{Theorem}[section]

\newtheorem{lemma}[theorem]{Lemma}
\newtheorem{proposition}[theorem]{Proposition}

\theoremstyle{definition}

\newtheorem{definition}[theorem]{Definition}
\newtheorem{example}[theorem]{Example}

\newtheorem{remark}[theorem]{Remark}

\makeatletter
\def\figcaption{%
\H@refstepcounter{figure}%
\@dblarg{\@caption{figure}}}
\makeatother

\makeatletter
\def\tablecaption{%
\H@refstepcounter{table}%
\@dblarg{\@caption{table}}}
\makeatother

\definecolor{mygreen}{RGB}{3,160,74}

\begin{document}
\sloppy
\renewcommand{\arraystretch}{1.5}

\title{Quantized Howe-type dualities via Koornwinder polynomials and the $X=K$ phenomenon}

\begin{abstract}
We derive the equality between one-dimensional sums associated with tensor
products of Kirillov-Reshetikhin column crystals of classical affine types and
Lusztig $q$-analogues of weight multiplicities. The matching of the corresponding root systems is suggested by Howe duality. Our main tool is the dual
Cauchy formula for Koornwinder polynomials due to Mimachi, which we combine with specializations in these polynomials. 
The mentioned dualities  are proved for  one-dimensional sums of all (twisted and untwisted) classical affine types except types $B_{n}^{(1)}$ and $D_{n}^{(1)}$. On another hand, all the Lusztig $q$-analogues of classical type are covered by our dualities, but they may have different parameters depending on the length of the roots in the underlying root system.
\end{abstract}

\author{Thomas Gerber}
\address{Institut Camille Jordan, Lyon 1 University, France}
\email{gerber@math.univ-lyon1.fr}

\author{Bogdan Ion}
\address{Department of Mathematics, University of Pittsburgh, USA}
\email{bion@pitt.edu}

\author{C\'{e}dric Lecouvey}
\address{Institut Denis Poisson, University of Tours, France}
\email{Cedric.Lecouvey@univ-tours.fr}

\author{Cristian Lenart}
\address{Department of Mathematics and Statistics,
State University of New York at Albany}
\email{clenart@albany.edu}

\maketitle

\section*{Introduction}

Lusztig's $q$-weight multiplicities are deformations of the usual weight multiplicities for complex simple Lie algebras~\cite{Lusztig1983}. 
They are also known as (generalized) Kostka-Foulkes (KF) polynomials, 
as in type $A$ they coincide with the usual KF polynomials. The numerous connections of these polynomials with many fundamental structures in representation theory make them particularly fascinating and challenging objects of study.   
Being affine Kazhdan-Lusztig polynomials \cite{Kato1982, Lusztig1983}, the $q$-weight multiplicities have non-negative integer coefficients,
but a combinatorial proof of this property is only known in type $A$ in full generality \cite{LS1978}, and relies on a statistic on semistandard Young tableaux called charge.
In fact, the approach in \cite{LS1978} was twofold, in the sense that a combinatorial proof of the positivity,
based on rank recursion (the so-called Morris recurrence formula), was given simultaneously with the combinatorial description in terms of semistandard tableaux. 
In other types, numerous partial results of this difficult problem have been established, but such a general picture remains incomplete for a general root system. 
We refer the reader, for example, to \cite{LN,JangKwon} for a description of the Kostka-Foulkes polynomials in classical types associated to the weight $0$ based on the combinatorics of generalized exponents, and for more historical background about the combinatorics of the Kostka-Foulkes polynomials with relevant references. 
We also mention~\cite{DGT,PT}, in which a charge statistic is shown to exist in type $C$ for row shapes and in rank 2, respectively.  

\medskip

The goal of this paper is to equate Kostka-Foulkes polynomials
(denoted $K_{\lambda,\mu}(q)$) with so-called one-dimensional sums (denoted $X_{\lambda,\mu}(q)$), 
which are polynomials in $q$ obtained as generating functions of the energy on the set of classical highest weight vertices of certain finite affine crystals known as Kirillov-Reshetikhin (KR) crystals.
This provides a duality between $q$-weight multiplicities and graded tensor product multiplicities, which, given the notation, is referred to as $X=K$.
The first $X=K$ result dates back to 1997, when Nakayashiki and Yamada \cite{NY1997} gave a combinatorial description of the Kostka-Foulkes polynomials of type $A$ by establishing a relationship between the Lascoux-Sch\"utzenberger charge and the energy function on a certain tensor product of column shape Kirillov-Reshetikhin crystals of affine type $A$.
In fact, these two statistics coincide up to a Howe-type duality, see \cite[Section 2.8]{GL2023} for a more modern viewpoint.
Later, a further $X=K$ result was derived in~\cite{LOS}, for all classical affine types in a stable limit. More precisely, it is shown that, if the two partitions $\lambda,\mu$ are fixed, and the rank $N$ of the corresponding root system goes to infinity, then the corresponding one-dimensional sum stabilizes (i.e., it does not depend on $N$). It is then proved that there are only four stable limits, which correspond to the following affine types: $A_N^{(1)}$, $C_N^{(1)}$, $D_N^{(1)}$, and $D_{N+1}^{(2)}$. Finally, the stable limits are realized as certain parabolic Lusztig $q$-weight multiplicities of the corresponding finite classical types. Much more recently, in \cite{CKL2024}, the Kostka-Foulkes polynomials of type $C_n$ (for any pair of dominant weights) and type $B_n$ (for any pair of spin dominant weights) are related to certain stable one-dimensional sums. Other versions of Lusztig $q$-weight multiplicities are also considered. 

\medskip

In this paper, we establish remarkable identities between the Lusztig $q$-weight multiplicities of all classical types and one-dimensional sums arising from the energy function on column Kirillov-Reshetikhin crystals of all (twisted and untwisted) classical affine types except type $B_n^{(1)}$ and $D_n^{(1)}$. 
More precisely, we prove that, for any  partitions $\la, \mu$ with at most $n$ rows and $m$ columns, we have
\begin{equation}\label{main-dual} X_{\la,\mu}(q) = K_{\widehat\la,\widehat\mu}(q) \,;\end{equation}
here 
$\widehat\la = (n-\la'_m, \ldots, n-\la'_1)$, with $\la'$ being the conjugate of $\la$, apart from affine type $A$, where such an identity is known, but we need to set $\widehat\la=\la'$. 
The matching of the corresponding root systems is suggested by Howe duality, and in fact the matched ranks are related to $m$ (for the $q$-weight multiplicities) and $n$ (for the one-dimensional sums). 
The general correspondence is more subtle than in type $A$ and, in some cases, it involves Kostka-Foulkes polynomials with unequal parameters. 
Furthermore, we sometimes need to change the labeling of the 0-node in the affine Dynkin diagram;  
although this operation does not change the affine root system up to isomorphism, it changes the one-dimensional sum considered. 
\Cref{table} illustrates the correspondence underlying the quantized duality results established in this paper. 
\begin{table}[H]
\centering
\begin{tabular}
[c]{lll}
\hline
Type of Kostka-Foulkes polynomial & Kostka-Foulkes parameter & Type of one-dimensional sum
\\
\hline
$A_{m-1}$ & $q$ & $A_{n-1}^{(1)}$\\
$C_{m}$ &$(q,q)$ & $A_{2n-1}^{(2)}$\\
$B_{m}$, integer weights & $(-q,q^{2})$ & $D_{n+1}^{(2)}$\\
$B_{m}$, half-integer weights& $(q,q^{2})$ & $A_{2n}^{(2)}%
$\\
$D_{m}$, integer weights &  $q$ & $A_{2n-1}^{(2,\dagger)}$\\
$D_{m}$, half-integer weights & $q$ & $A_{2n}^{(2,\dagger)}%
$\\
$C_{m}$ & $(0,q)$ & $C_{n}^{(1)}$\\
\hline
\\
\end{tabular}    
\tablecaption{Matching types between Kostka-Foulkes polynomials and one-dimensional sums.}
\label{table}
\end{table}

We discuss some computational applications of our results, while referring to Section~\ref{sec:final} for additional ones. We can derive a combinatorial description of the classical Kostka-Foulkes polynomials in terms of the corresponding affine crystals, cf.~\Cref{table}. Indeed, while the definition of the energy function via local energies is impractical, it was shown in~\cite{LNSSS} that, on a tensor product of column shape Kirillov-Reshetikhin crystals, this function can be computed very explicitly in terms of a type-independent combinatorial model known as the quantum alcove model. This is based on a directed graph on the corresponding Weyl group known as the quantum Bruhat graph. Moreover, in~\cite{Lenart,LeSc} it was shown that, in all classical types, the mentioned computations can be pushed to the corresponding (type-specific) tableau models (based on Kashiwara-Nakashima columns). In fact, in~\cite{Lenart} it was also shown that, by applying the same procedure in type $A$, one easily rederives the Lascoux-Sch\"{u}tzenberger charge statistic on semistandard tableaux.

We will now compare the results obtained in the present paper and those in~\cite{CKL2024}. 
\begin{enumerate}
\item \Cref{table} permits to equate, for any pair of dominant weights of a given classical root system 
(partitions or half-integer partitions in orthogonal types), its associated Kostka-Foulkes polynomial with a one-dimensional sum. \cite{CKL2024} covers the two types of classical Kostka-Foulkes polynomials mentioned above. 
Our level of generality sometimes requires us to consider Kostka-Foulkes polynomials with unequal parameters, possibly negative. In fact, the $(q,t)$-Kostka-Foulkes polynomials of type $B$ corresponding to a pair of half-integer partitions (spin weights) were already considered in~\cite{CKL2024}, whereas we consider their specialization at $t=q^2$. 
\item Our dualities are direct $X=K$ ones, whereas those in~\cite{CKL2024} sometimes involve only a certain part of a one-dimensional sum, which is identified via subtle combinatorics.
\item The duality results in \cite{CKL2024} hold for one-dimensional sums considered in large rank. 
In our work, we match root systems of arbitrary finite ranks (as suggested by Howe duality), without any assumption on these ranks being large. 
In order to understand the relationship between these two types of dualities, first recall from~\cite{LOS} that, for $n$ large enough, 
the one-dimensional sums of type $B_n^{(1)}$ and $D_{n+1}^{(2)}$ (those appearing in~\cite{CKL2024})
coincide respectively with those of type $A_{2n-1}^{(2)}$ and $A_{2n}^{(2)}$ (appearing in the present paper).
When this happens, while increasing $n$ and keeping $m$ fixed in~\eqref{main-dual}, we recover the results of \cite{CKL2024} involving certain, but not all, Kostka-Foulkes  polynomials of type $C$ and type $B$ (for spin weights and $t=q^2$).
\item We only need to consider one-dimensional sums associated to tensor products of column shape Kirillov-Reshetikhin-crystals. Tensor products of row shape Kirillov-Reshetikhin crystals are also considered in \cite{CKL2024} in the stable case, in relation to so-called level-restricted $q$-weight multiplicities.
% Although there is a similarity between the results in~\cite{CKL2024} and those in this paper for the Kostka-Foulkes polynomials of types $C$ and $B$ (in the latter case, for half-integer weights), there is no direct implication in either direction. 
\item The methods used in this paper are completely different from those in  \cite{CKL2024}, which are based on the intricate tableau combinatorics in classical types and the Morris-type recurrence formulas for the Kostka-Foulkes polynomials in~\cite{Lecouvey2006}. Our approach relies on properties of Macdonald-Koornwinder polynomials, which lead to more concise and conceptual proofs. 
\end{enumerate}

More precisely, our proofs are based on known Cauchy-type identities and connections between Macdonald polynomials specialized at $t=0$ and Kirillov-Reshetikhin crystals. 
In particular, we rederive Nakayashiki and Yamada's type $A$ result directly from the theory of Macdonald polynomials, without using any combinatorial description of the charge or the energy statistic.

\medskip

The structure of the paper is as follows.
In \Cref{sec_gen} we recall  the background on root systems and the Weyl  characters relevant for our purposes. \Cref{sec_Koornwinder} is devoted to Koornwinder polynomials, the dual Cauchy formula they satisfy, as well as their connections with Macdonald polynomials, Hall-Littlewood polynomials, affine Demazure characters, and Weyl module characters. Then, in~\Cref{sec_type_A}, we explain the way in which the classical type $A$ identity equating Kostka-Foulkes polynomials and one-dimensional sums  \cite{NY1997} can be recovered based on the usual Cauchy identities (for Schur functions, Hall-Littlewood polynomials, and Macdonald polynomials). This section is relatively independent of the others and should help the reader understand our general strategy. In \Cref{sec_type_C}, we adapt the previous ideas in order to derive our main result in type $C_{n}$. The case of Kostka-Foulkes polynomials of type $D_m$ parametrized by a pair of partitions is examined in \Cref{type_D_int}. In \Cref{sec_type_BD}, we use Kostka-Foulkes polynomials of type $B_m$ with  unequal parameters indexed by a pair of half-integer partitions or a pair of partitions, and relate them to certain one-dimensional sums. Here the second case requires the use of a negative parameter. Finally, we study the case of type $D_m$ Kostka-Foulkes polynomials parametrized by pairs of half-integer partitions in \Cref{sec_type_D_halfint}. All these identities require more work than in the type $A$ case, as we need to use a Cauchy identity at the level of Koornwinder polynomials and study specific ``non-Macdonald'' specializations. Furthermore, although the general strategy is the same, its realization depends on the type considered. Therefore, for the clarity of the exposition, we study each case separately. \Cref{sec_nontwisted} of the paper is devoted to the inverse problem of equating any one-dimensional sum associated with a tensor product of column Kirillov-Reshetikhin crystals of a given classical type with a Kostka-Foulkes polynomial. We show that this is indeed possible in type $C_n^{(1)}$, but present some obstructions in the remaining untwisted classical types. Nevertheless, we believe that it is possible to equate the remaining one-dimensional sums of types $B_n^{(1)}$ and $D_n^{(1)}$ with generalizations of Kostka-Foulkes polynomials, and we are currently working on this problem. We present a worked example in Section \ref{SecExample}, which is carried out in each affine type. Our final section outlines several directions for future research motivated by the results of this paper.   

\subsection*{Acknowledgment} 
T. Gerber and C. Lecouvey were partially supported by the Agence Nationale de la Recherche funding ANR CORTIPOM 21-CE40-0019. 
B. Ion was partially supported by the Simons Foundation grant 420882. 
C. Lenart was partially supported by the NSF grant DMS-2401755.

\bigskip\noindent2010 Mathematics Subject Classification. 05E10, 17B10.

\section{Background on representation theory and root systems}
\label{sec_gen}

\subsection{Simple Lie algebras and finite root systems}
\label{Section_DD}

In this section, we recall some classical results on root
systems and the representation theory of the Lie algebras over $\mathbb{C}$.
We refer the reader to \cite{BBK,FH,Hum} for a detailed
exposition. Consider such a finite-dimensional simple algebra $\mathfrak{g}^{T_{n}}$ with root system $R^{T_{n}}$ 
of type $T_n$, realized in the
Euclidean space $E=\oplus_{i=1}^{n}\mathbb{R\varepsilon}_{i}$. 
When there is no risk of confusion, we will drop the superscript $T_{n}$ to simplify the
notation and simply write for example $\mathfrak{g},R$ instead of
$\mathfrak{g}^{T_{n}},R^{T_{n}}$.
The Dynkin diagram of $R$ is indexed
by $I=\{1,\ldots,n\}$ and we denote as usual by

\begin{itemize}
\item $R_{+}$ and $S=\{\alpha_{i},i\in I\}$ the subsets of positive and simple roots respectively,

\item $W$ the Weyl group with generators $s_{i},\;i\in I$ associated with the
simple roots $\alpha_{i},i\in I$,

\item $\ell$ the length function on $W$: for any $w$ in $W$, $\ell(w)$ is the
number of generators $s_{i}$ in any reduced expression of $w$,

\item $Q$ the root lattice and $Q_{+}$ the cone generated by the positive roots,

\item $P$ the weight lattice and $P_{+}$ the cone of dominant weights,
generated by the fundamental weights $\omega_{i},i\in I$,

\item $\rho=\sum_{i=1}^{n}\omega_{i}=\frac{1}{2}\sum_{\alpha\in R_{+}}\alpha$,
the half-sum of positive roots,

\item $V(\nu)$ the simple $\mathfrak{g}$-module of highest weight $\nu\in
P_{+},$

\item $\leq$ the dominance order on $P$, defined by $\gamma\leq\mu$ if and only
if $\mu-\gamma\in Q_{+}$.
\end{itemize}

We also recall the Weyl character formula. For each dominant weight $\lambda$ in
$P_{+}$, the character of $V(\lambda)$ is the polynomial $s_\lambda\in\mathbb{Z}^{W}[P]=\{U\in\mathbb{Z}[P]\mid w(U)=U\}$ verifying
\[
s_{\lambda}=\frac{\sum_{w\in W}\varepsilon(w)e^{w(\lambda+\rho)-\rho}}{\prod_{\alpha\in R_{+}}(1-e^{-\alpha})}.
\]
We shall also use the notation $a_{\gamma}=\sum_{w\in W}\varepsilon
(w)e^{w(\gamma)}$ for any $\gamma\in P$. Then $s_{\lambda}=\frac
{a_{\lambda+\rho}}{a_{\rho}}$. The family of polynomials $\{m_{\mu}\mid\mu\in
P_{+}\}$ where
\[
m_{\mu}=\sum_{\gamma\in W\cdot\mu}e^{\gamma}
\]
is another basis of the character ring $\mathbb{Z}^{W}[P]$. The generalized
Kostka numbers are the coefficients in the expansion of the Weyl characters on
this basis:
\begin{equation}
s_{\lambda}=\sum_{\mu}K_{\lambda,\mu}\,m_{\mu}. \label{Kostka}
\end{equation}
The generalized Kostka number $K_{\lambda,\mu}$ is a nonnegative integer equal
to the dimension of the weight space of weight $\mu$ in the representation
$V(\lambda)$.

We will be interested in the classical root systems of type $A_{n-1}
,B_{n},C_{n}$ and $D_{n}$. We will assume the classical realization of these
root systems, namely%
\[
S=\left\{
\begin{array}
[c]{l}%
\{\alpha_{i}=\varepsilon_{i}-\varepsilon_{i+1},\;i=1,\ldots,n-1\}\text{ in type
}A_{n-1}\\
\{\alpha_{i}=\varepsilon_{i}-\varepsilon_{i+1},\;i=1,\ldots,n-1\text{ and
}\alpha_{n}=\varepsilon_{n}\}\text{ in type }B_{n}\\
\{\alpha_{i}=\varepsilon_{i}-\varepsilon_{i+1},\;i=1,\ldots,n-1\text{ and
}\alpha_{n}=2\varepsilon_{n}\}\text{ in type }C_{n}\\
\{\alpha_{i}=\varepsilon_{i}-\varepsilon_{i+1},\;i=1,\ldots,n-1\text{ and
}\alpha_{n}=\varepsilon_{n+1}+\varepsilon_{n}\}\text{ in type }D_{n},
\end{array}
\right.
\]
and
\[
R_{+}=\left\{
\begin{array}
[c]{l}%
\{\varepsilon_{i}-\varepsilon_{j},1\leq i<j\leq n\}\text{ in type }A_{n-1}\\
\{\varepsilon_{i}\pm\varepsilon_{j},\;1\leq i<j\leq n\text{ and }\varepsilon
_{i},i=1,\ldots n\}\text{ in type }B_{n}\\
\{\varepsilon_{i}\pm\varepsilon_{j},\;1\leq i<j\leq n\text{ and }\varepsilon
_{i},i=1,\ldots n\}\text{ in type }C_{n}\\
\{\varepsilon_{i}\pm\varepsilon_{j},\;1\leq i<j\leq n\}\text{ in type }D_{n}.%
\end{array}
\right.
\]
Also, it will be convenient to set $x_{i}=e^{\varepsilon_{i}}$ in order to
identify the previous character ring $\mathbb{Z}^{W}[P]$ with the ring of
symmetric polynomials for type $A_{n-1}$ in the indeterminates $x_{1},\ldots,x_{n}$ (where we will consider for simplicity the Lie algebra
$\mathfrak{gl}_{n}$ rather than $\mathfrak{sl}_{n}$) and with the ring of
symmetric Laurent polynomials for types $B_{n}$, $C_{n}$, or $D_{n}$.

\subsection{Affine Lie algebras and crystals}

Recall here that the affine root systems were classified by Kac (see
\cite{Kac}) in terms of their associated affine Dynkin diagram. Each such
Dynkin diagram of type $T_{N}^{(a)}$ is obtained by adding an affine node
(usually labelled by $0$) to one of the Dynkin diagrams associated with a
finite root system of rank $n$. 
Here again, in what follows, we will only use a
superscript $T_{N}^{(a)}$ when it will be required for the clarity of the
exposition. The root lattice so obtained is then
\[
Q_{\mathrm{a}}=Q\oplus\mathbb{Z\delta'=}
{\textstyle\bigoplus\limits_{i=1}^{n}}
\mathbb{Z\alpha}_{i}\oplus\mathbb{Z\delta'}\,,
\]
where $\delta'=\delta$ is
the imaginary null root for any affine classical types but type $A_{2n}^{(2)}$ where $\delta'=\frac{1}{2}\delta$.
The affine weight lattice can then be described as
\[
P_{\mathrm{a}}=\mathbb{Z}\Lambda_{0}\oplus P\oplus\mathbb{Z\delta'=Z}
\Lambda_{0}\oplus%
{\textstyle\bigoplus\limits_{i=1}^{n}}
\mathbb{Z\omega}_{i}\oplus\mathbb{Z\delta'}\,,
\]
where $P=\oplus_{i=1}^{n}\mathbb{Z\omega}_{i}$ is the weight lattice of the
finite root system with fundamental weights $\omega_{i},i=1,\ldots n$ and
$\Lambda_{0}$ the fundamental affine weight associated with the $0$-node.
We
will denote by $W_{\mathrm{a}}$ the affine Weyl group associated with our
affine root system. It is generated by the affine reflections $s_{i},\:i=0,1,\ldots n$ and contains the finite Weyl group $W$ as the subgroup
generated by the $s_{i},i=1,\ldots,n$. In the rest of this paper, we will
assume that the affine root systems that we consider are of classical type,
that is their underlying finite root system is of type $A,B,C$ or $D$.
Kac's classification ensures that
$T_{N}^{(a)}$ is one of 
$A_{n-1}^{(1)}$, $B_{n}^{(1)}$, $C_{n}^{(1)}$, $D_{n}^{(1)}$ (untwisted types),
$A_{2n}^{(2)}$, $A_{2n-1}^{(2)}$, $D_{n+1}^{(2)}$ (twisted types). 
In fact, we will also need the variations $A_{2n-1}^{(2,\dag)}$, $A_{2n}^{(2,\dag)}$, and $B_{n}^{(1,\dag)}$ of the
affine root systems $A_{2n-1}^{(2)}$, $A_{2n}^{(2)}$, and $B_{n}^{(1)}$ respectively, in which we relabel the nodes of the Dynkin diagrams by changing each label $i$ into $n-i$. 
This does not change the associated root system up to isomorphism but will change the 
energy statistic.
\Cref{fig:aff} and \Cref{fig:aff_twist} contain the affine Dynkin diagrams that will be
used in this paper.

\begin{figure}[ht]
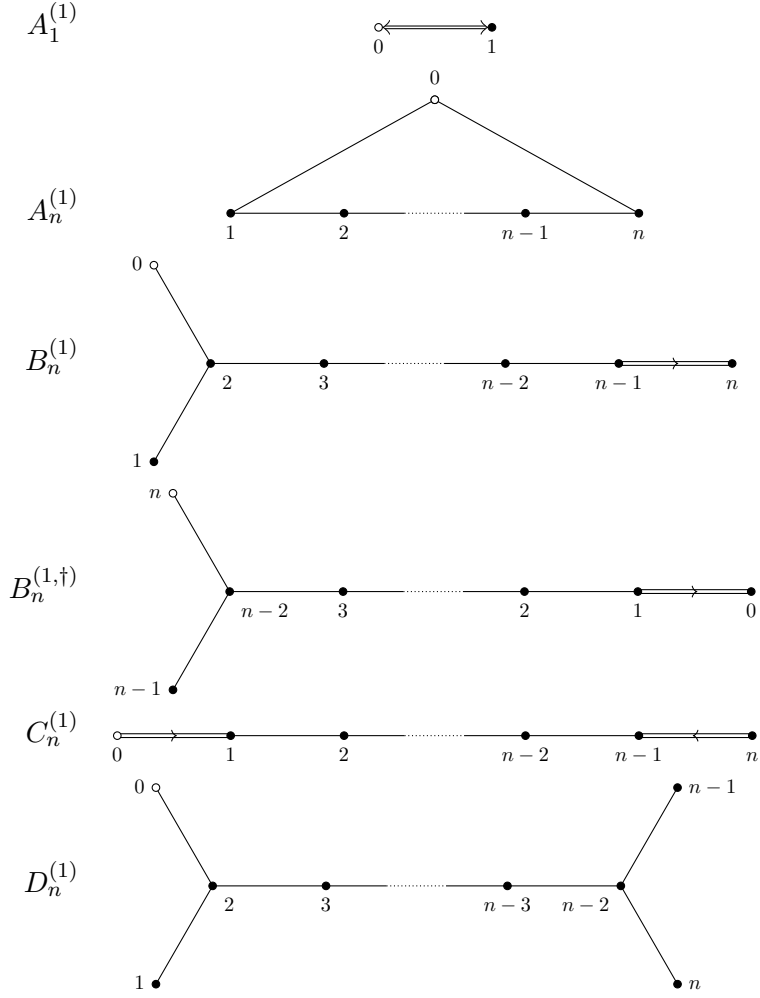

\begin{tabular}
[c]{rc}%
$A_{1}^{(1)}$ & \dynkin[labels={0,1}, edge length=1.5cm, extended] A[1]{1}\\
$A_{n}^{(1)}$ & \dynkin[labels={0,1,2,n-1,n}, edge length=1.5cm, extended]
A[1]{}\\
$B_{n}^{(1)}$ &
\dynkin[labels={0,1,2,3,n-2,n-1,n}, edge length=1.5cm, extended] B[1]{}\\
$B_{n}^{(1,\dagger)}$ &
\dynkin[labels={n,n-1,n-2,3,2,1,0}, edge length=1.5cm, extended] B[1]{}\\
$C_{n}^{(1)}$ & \dynkin[labels={0,1,2,n-2,n-1,n}, edge length=1.5cm, extended]
C[1]{}\\
$D_{n}^{(1)}$ &
\dynkin [labels={0,1,2,3,n-3,n-2,n-1,n}, edge length=1.5cm, extended] D[1]{}\\
&
\end{tabular}
\figcaption{Affine Dynkin diagrams of untwisted classical types.}
\label{fig:aff}
\end{figure}

\begin{figure}[ht]
\begin{tabular}
[c]{rc}%
$A_{2}^{(2)}$ & \dynkin[labels={0,1}, edge length=1.5cm, extended] A[2]{2}\\
$A_{2n}^{(2)}$ &
\dynkin[labels={0,1,2,3,n-2,n-1,n}, edge length=1.5cm, extended] A[2]{even}\\
$A_{2n}^{(2,\dagger)}$ &
\dynkin[labels={n,n-1,n-2,n-3,2,1,0}, edge length=1.5cm, extended]
A[2]{even}\\
$A_{2n-1}^{(2)}$ &
\dynkin[labels={0,1,2,3,4,n-2,n-1,n}, edge length=1.5cm, extended] A[2]{odd}\\
$A_{2n-1}^{(2,\dagger)}$ &
\dynkin[labels={n,n-1,n-2,n-3,n-4,2,1,0}, edge length=1.5cm, extended]
A[2]{odd}\\
$D_{n+1}^{(2)}$ &
\dynkin[labels={0,1,2,3,n-3,n-2,n-1,n}, edge length=1.5cm, extended] D[2]{}\\
&
\end{tabular}
\figcaption{Affine Dynkin diagrams of twisted classical types.}
\label{fig:aff_twist}
\end{figure}

\medskip

Recall that crystal graphs can be regarded as combinatorial skeletons of
irreducible highest weight modules associated with any simple affine Lie
algebra $\widehat{\mathfrak{g}}$, see \cite{HongKang, BumpSchilling} for generalities.
More precisely, in the Kashiwara-Lusztig approach
of crystal theory, these objects are defined from representation theory of the quantum group
$U_{\nu}(\widehat{\mathfrak{g}}\mathfrak{)}$ associated with $\widehat
{\mathfrak{g}}$.
The irreducible highest weight $U_{\nu}(\widehat{\mathfrak{g}}\mathfrak{)}$-modules
are parametrized by the dominant affine weights: write $V(\Lambda)$ for
the module labeled by the dominant weight $\Lambda$. 
The crystal $B(\Lambda)$
is then an oriented graph with arrows $\overset{i}{\rightarrow}$, where 
$i\in \{0,1,\ldots,n\}$, equipped with a weight map
\[
\mathrm{wt}:B(\Lambda)\rightarrow P_{\mathrm{a}}.
\]
One can then compute the character of $V(\Lambda)$ as the generating series
of the weight map, that is
\[
\mathrm{char}(V(\Lambda))=\sum_{b\in B(\Lambda)}e^{\mathrm{wt}(b)}.
\]
This crystal $B(\Lambda)$ has a unique source vertex $b_{\Lambda}$ of weight
$\Lambda$ and a natural grading: the degree $\mathrm{d}(b)$ of $b$ in
$B(\Lambda)$ is the number of $0$-arrows in any path connecting $b_{\Lambda}$
to $b$.
A remarkable property of crystals is its compatibility with the tensor product.
More precisely, the decomposition of a tensor product of
simple modules into irreducible components is obtained by looking at the
decomposition of the associated crystal into its connected components.

\medskip

There is also a rich finite-dimensional representation theory of a
particular subalgebra $U_{\nu}^{\prime}(\widehat{\mathfrak{g}})\subset U_{\nu
}(\widehat{\mathfrak{g}})$.
The associated simple modules are no longer of highest
weight and the associated category is not semisimple.
In this article, we are
especially interested in some particular simple finite-dimensional such modules
called the Kirillov-Reshetikhin modules. These are
parametrized by pairs $(r,s)\in\{1,\ldots,n\}\times\mathbb{Z}_{>0}$, and are denoted $B^{(r,s)}$. In what follows, we will only consider these KR modules for $s=1$, which are called
\textit{column KR modules} since $r$ can be
interpreted as the height of a column-shaped Young diagram.\ 
The KR modules also have
associated crystal graphs with $i$-arrows indexed by $\{0,1,\ldots ,n\}$.\ 
These graphs are finite and connected but do not have a distinguished
source vertex (reflecting the fact the KR modules are finite-dimensional but
not of highest weight). 
When $s=1$, they admit a simple
combinatorial description in terms of column tableaux introduced by Kashiwara
and Nakashima \cite{KN1994}, which depends on the classical affine type
considered, see \cite{BumpSchilling} for more details.

\begin{example}
In type $A_{n-1}^{(1)}$, the vertices of $B^{(r,1)}$ can be identified with
the column tableaux of height $r$ on the alphabet $\{1,\ldots,n\}$. For
$i=1,\ldots,n-1$, there is an edge
$C \overset{i}{\rightarrow} C^{\prime}$
if and only if there exists $i\in\{1,\ldots,n-1\}$ such that $i\in C$, $i+1\notin
C^{\prime}$ and $C^{\prime}$ is the column obtained by replacing $i$ by $i+1$
in $C$; there is an edge $C\overset{0}{\rightarrow}C^{\prime}$ if $n\in C$,
$1\notin C^{\prime}$ and $C^{\prime}$ is the column obtained by replacing $n$
by $1$ in $C$ and by reordering the entries.
For instance, take $n=5$ and $r=2$. The set of vertices in the column KR crystal
$B^{(2,1)}$ is
\[%
\begin{array}
[c]{ccccccccccc}%
\scalebox{.8}{\begin{ytableau} 1 \\ 2\end{ytableau}} &
\scalebox{.8}{\begin{ytableau} 1 \\ 3\end{ytableau}} &
\scalebox{.8}{\begin{ytableau} 2 \\ 3\end{ytableau}} &
\scalebox{.8}{\begin{ytableau} 1\\ 4\end{ytableau}} &
\scalebox{.8}{\begin{ytableau} 2 \\ 4\end{ytableau}} &
\scalebox{.8}{\begin{ytableau} 1 \\ 5\end{ytableau}} &
\scalebox{.8}{\begin{ytableau} 3 \\ 4\end{ytableau}} &
\scalebox{.8}{\begin{ytableau} 2 \\ 5\end{ytableau}} &
\scalebox{.8}{\begin{ytableau} 3 \\ 5\end{ytableau}} &
\scalebox{.8}{\begin{ytableau} 4 \\ 5\end{ytableau}} & 
\end{array},
\]
and the column KR crystal $B^{(2,1)}$ is shown in~\Cref{fig:crystal}.

\begin{figure}[ht]
\begin{center}
\begin{tikzpicture}[scale=1]%[H]
				\tikzstyle{vertex}=[inner sep=2pt,minimum size=10pt]
				
				\node[vertex] (a0) at (0,0) {$\scalebox{.7}{\begin{ytableau} 1 \\ 2\end{ytableau}}$};
				
				\node[vertex] (a1) at (0,-1.5) {$\scalebox{.7}{\begin{ytableau} 1 \\ 3\end{ytableau}}$};
				
				\node[vertex] (a2) at (-1,-3) {$\scalebox{.7}{\begin{ytableau} 2 \\ 3\end{ytableau}}$};
				\node[vertex] (a3) at (1,-3) {$\scalebox{.7}{\begin{ytableau} 1 \\ 4\end{ytableau}}$};
				
				\node[vertex] (a4) at (-1,-4.5) {$\scalebox{.7}{\begin{ytableau} 2 \\ 4\end{ytableau}}$};
				\node[vertex] (a5) at (1,-4.5) {$\scalebox{.7}{\begin{ytableau} 1 \\ 5\end{ytableau}}$};
				
				\node[vertex] (a6) at (-1,-6) {$\scalebox{.7}{\begin{ytableau} 3 \\ 4\end{ytableau}}$};
				\node[vertex] (a7) at (1,-6) {$\scalebox{.7}{\begin{ytableau} 2 \\ 5\end{ytableau}}$};
				
				\node[vertex] (a8) at (0,-7.5) {$\scalebox{.7}{\begin{ytableau} 3 \\ 5\end{ytableau}}$};
				
				\node[vertex] (a9) at (0,-9) {$\scalebox{.7}{\begin{ytableau} 4 \\ 5\end{ytableau}}$};

				\draw[line width=1pt,->] (a0)  -- node[right] {\footnotesize $2$}  (a1);
				\draw[line width=1pt,->] (a1)  -- node[right] {\footnotesize $1$}  (a2);
				\draw[line width=1pt,->] (a1)  -- node[right] {\footnotesize $3$} (a3);
				\draw[line width=1pt,->] (a2)  -- node[right] {\footnotesize $3$} (a4);
				\draw[line width=1pt,->] (a3)  -- node[right] {\footnotesize $1$} (a4);
				\draw[line width=1pt,->] (a3)  -- node[right] {\footnotesize $4$} (a5);
				\draw[line width=1pt,->] (a4)  -- node[right] {\footnotesize $2$} (a6);
				\draw[line width=1pt,->] (a4)  -- node[right] {\footnotesize $4$} (a7);
				\draw[line width=1pt,->] (a5)  -- node[right] {\footnotesize $1$} (a7);
				
				\draw[line width=1pt,->] (a6)  -- node[right] {{\footnotesize $4$}} (a8);
				\draw[line width=1pt,->] (a7)  -- node[right] {\footnotesize $2$} (a8);
				
				\draw[line width=1pt,->] (a8)  -- node[right] {\footnotesize $3$} (a9);

                \draw[->,line width=0.5pt]
                (a7) to[out=0,in=0] node[right] {\footnotesize $0$} (a0);
                
                \draw[->,line width=0.5pt]
                (a8) to[out=180,in=-180] node[left] {\footnotesize $0$} (a1);
                
                \draw[->,line width=0.5pt]
                (a9) to[out=0,in=0] node[right] {\footnotesize $0$} (a3);
                                        
				\end{tikzpicture}
\end{center}
\figcaption{The column KR crystal $B^{(2,1)}$.}
\label{fig:crystal}
\end{figure}
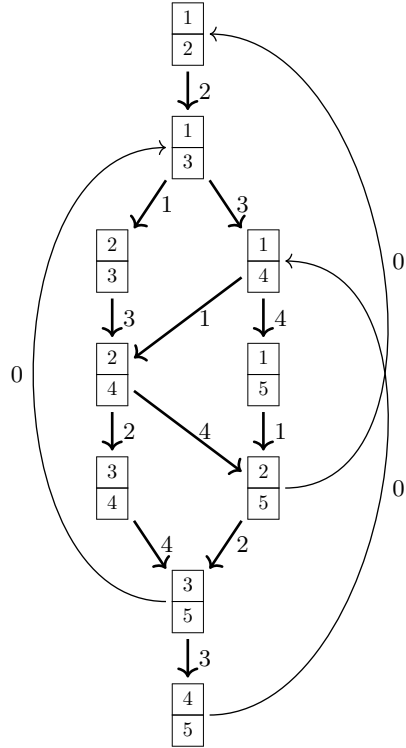
\end{example}

The tensor products of KR modules are still irreducible,
thus the tensor product of column
KR crystals give finite-connected crystals. 
These crystals do not remain
irreducible in general when one removes the $0$-arrows. In fact, these
$0$-arrows permit to define a subtle statistic $D$, called \textit{energy}, on any
tensor product $B$ of KR crystals in such a way that $D$ is constant on the
classical components of $B$ (obtained by removing the $0$-arrows). The
definition of $D$ is quite involved and also depends on a choice of normalization 
(i.e. the choice of a particular classical component where $D$ is zero).
For the purpose of this paper, we can in fact bypass this definition by
exploiting the connection between our crystal $B$ and a well-chosen affine
Demazure crystal, which we will explain in \Cref{Th_ST}. 
For the moment, let us denote by $\mathrm{HW}(B)$ the set of classical highest weight vertices in $B$, that is the set of
vertices in $b$ with no incident $i$-arrows but one $0$-arrow. Each vertex $b$
in $\mathrm{HW}(B)$ admits a classical dominant weight$.$ 
We can then collect
all the classical highest weight vertices in $\mathrm{HW}(B)$ with prescribed dominant weight $\lambda$. 
The generating series associated with the energy
function $D$ over the subset $\mathrm{HW}(B)_{\lambda}$ of $\mathrm{HW}(B)$
with dominant weight $\lambda$ is called a one-dimensional sum (1-d sum for
short). In this paper we are particularly interested in the 1-d sums defined
from a partition $\mu$ in the rectangular partition $(m^{n})$.
(whose Young diagram has $n$ rows and $m$ columns). 
To such a partition, we can indeed associate the tensor product of
column KR crystals%
\[
B_{\mu}=B^{(\mu_{1}^{\prime},1)}\otimes\cdots\otimes B^{(\mu_{m}^{\prime},1)}\,,
\]
where $\mu_{1}^{\prime},\ldots,\mu_{m}^{\prime}$ are the columns of the Young
diagram associated with $\mu$.

\begin{definition}
\label{Def1-d}
Let $\mu$ be a partition in the rectangle $(m^n)$.
The \textit{one-dimensional sum} associated to $\mu$ and to the dominant weight $\la$ is the polynomial
\[
X_{\lambda,\mu}(q)=\sum_{b\in\mathrm{HW}(B_{\mu})_{\lambda}}q^{D(b)}.
\]
\end{definition}

\section{Koornwinder polynomials and their Macdonald specializations}
\label{sec_Koornwinder}

\subsection{Basics on Koornwinder polynomials}

In this paragraph, we recall the definition of the Koornwinder polynomials.
We refer the reader to \cite{Mimachi2001} for a brief history and more details.\
Consider a formal parameter $a$ and recall the notation for the $q$-Pochhammer symbol%
\begin{equation*}
(a;q)_{\infty }=\prod_{j\geq 0}(1-aq^{j}).
\end{equation*}
Similarly for a family of formal parameters $(a_{1},\ldots ,a_{m})$ set%
\begin{equation*}
(a_{1},\ldots ,a_{m};q)_{\infty }=\prod_{k=1}^{m}(a_{k};q)_{\infty }.
\end{equation*}%
In what follows $q,t,a,b,c,d$ are indeterminates and we set $\mathbb{K}=%
\mathbb{C}(q,t,a,b,c,d)$. Consider the ring $\mathbb{K}^{W}[x_{1}^{\pm
1},\ldots ,x_{n}^{\pm 1}]$ of Laurent polynomials in the indeterminates $x_{1},\ldots ,x_{n}$ invariant by the action of $W,$ the Weyl group of type $C_{n}$.\ 
As mentioned in the previous section, this can be regarded as the character ring of type $C_{n}$ over $%
\mathbb{K}$. In particular $\mathbb{K}^{W}[x_{1}^{\pm 1},\ldots ,x_{n}^{\pm
1}]$ admits various bases indexed by the set $\mathcal{P}_{n}$ of partitions
with at most $n$ parts (which here can be identified with the set of
dominant weights of type $C_{n}$).\ We have a natural bar involution on $\mathbb{K}[x_{1}^{\pm 1},\ldots ,x_{n}^{\pm 1}]$ such that for any $f$ in $\mathbb{K}[x_{1}^{\pm 1},\ldots ,x_{n}^{\pm 1}]$, the polynomial $\overline{f}$ is obtained by replacing each $x_{i}$ by its inverse $x_{i}^{-1}$. We
write $[f]_{1}$ for the constant term in $f$.\ Now we can define a paring on 
$\mathbb{K}^{W}[x_{1}^{\pm 1},\ldots ,x_{n}^{\pm 1}]$ by 
\begin{equation*}
\langle f,g\rangle =[f\Delta \overline{g\Delta }]_{1}
\end{equation*}%
where 
\begin{equation*}
\Delta =\prod_{1\leq i<j\leq n,~s_{1},s_{2}\in \{\pm 1\}}\frac{%
(x_{i}^{s_{1}}x_{j}^{s_{2}};q)_{\infty }}{(tx_{i}^{s_{1}}x_{j}^{s_{2}};q)_{%
\infty }}\prod_{j=1,s\in \{\pm 1\}}^{n}\frac{(x_{j}^{2s};q)_{\infty }}{%
(ax_{j}^{s},bx_{j}^{s},cx_{j}^{s},dx_{j}^{s};q)_{\infty }}.
\end{equation*}%
The following proposition defines the basis of Koornwinder polynomials.

\begin{proposition}
There exists a unique basis $\{P_{\lambda }(x;a,b,c,d;q,t)\mid \lambda \in 
\mathcal{P}_{n}\}$ of $\mathbb{K}^{W}[x_{1}^{\pm 1},\ldots ,x_{n}^{\pm 1}]$
whose decomposition on the monomial basis $\{m_{\lambda }(x)\mid \lambda \in 
\mathcal{P}_{n}\}$ is unitriangular for the dominant order and such that 
\begin{equation*}
\langle P_{\lambda }(x;a,b,c,d;q,t),P_{\mu }(x;a,b,c,d;q,t)\rangle =0
\end{equation*}%
for any pair $\lambda,\mu \in \mathcal{P}_{n}$ with that $\lambda
\neq \mu $.\ The polynomials $P_{\lambda }(x;a,b,c,d;q,t)$, for $\lambda \in 
\mathcal{P}_{n}$, are called the Koornwinder polynomials (of rank $n$).
\end{proposition}

\begin{remark}
The Koornwinder polynomials are invariant with respect to the permutation of the
indeterminates $a,b,c,d$ which are called \emph{Askey-Wilson parameters}.\ They can also be defined as the eigenfunctions of a
remarkable order one $q$-difference operator acting on $\mathbb{K}^{W}[x_{1}^{\pm 1},\ldots,x_{n}^{\pm 1}]$.
\end{remark}

\subsection{Koornwinder polynomials and the dual Cauchy identity}

Fix $n,m\in\mathbb{Z}_{\geq1}$. We denote by $(m^{n})$ the rectangular
partition with $n$ rows and $m$ columns. For any partition $\lambda
\subseteq(m^{n})$, we set
\[
\widehat{\lambda}=(n-\la'_m, \ldots, n-\la'_1)\,,
\]
where $\lambda^{\prime}=(\lambda_{1}^{\prime},\ldots,\lambda_{m}^{\prime
})\subseteq(n^{m})$ is the conjugate of $\lambda$.

\noindent In \cite[Theorem 2.1]{Mimachi2001}, Mimachi establishes a
dual Cauchy-type identity between rank $n$ and rank $m$  Koornwinder polynomials, which is stated as follows:
\begin{equation}
\prod_{\substack{1\leq i\leq n\\1\leq j\leq m}}(x_{i}+x_{i}^{-1}-y_{j}-y_{j}^{-1})=\sum_{\lambda\subseteq(m^{n})}(-1)^{|\widehat{\lambda}|}P_{\lambda
}(x;a,b,c,d;q,t)\,P_{\widehat{\lambda}}(y;a,b,c,d;t,q). \label{Cauchy_Koorn}
\end{equation}
Note that $P_{\lambda}(x;a,b,c,d;q,t)$ is a rank $n$  Koornwinder
\emph{$(q,t)$-polynomial} in $x=x_{1},\ldots,x_{n}$, and $P_{\widehat{\lambda}}(y;a,b,c,d;t,q)$ is a rank $m$  Koornwinder
\emph{$(t,q)$-polynomial} in $y=y_{1},\ldots,y_{m}$, both with the same  Askey-Wilson parameters
$a,b,c,d$.

\subsection{Macdonald specializations}

It is known that these Koornwinder polynomials recover Macdonald polynomials
by appropriate specialization of the Askey-Wilson parameters $a,b,c,d$. More precisely, we
recover Macdonald polynomials $P_{\lambda}^{T_{N}^{(a)}}(x; q,t,u)$ of each
classical affine types $T_{N}^{(a)}$ by using the specializations in Table~\ref{table_spec}, see e.g. \cite{DFK25}.

\begin{table}[ht]%
\[%
\begin{array}
[c]{@{}l@{\hskip20pt}l@{\hskip20pt}l@{\hskip20pt}l@{\hskip20pt}l}\hline
\text{Macdonald type} & a & b & c & d\\\hline
D_{n}^{(1)} & 1 & -1 & q^{\frac{1}{2}} & -q^{\frac{1}{2}}\\
B_{n}^{(1)} & u & -1 & q^{\frac{1}{2}} & -q^{\frac{1}{2}}\\
B_{n}^{(1,\dagger)} & 1 & -1 & q^{\frac{1}{2}}u & -q^{\frac{1}{2}}\\
C_{n}^{(1)} & u^{\frac{1}{2}} & -u^{\frac{1}{2}} & u^{\frac{1}{2}}q^{\frac
{1}{2}} & -u^{\frac{1}{2}}q^{\frac{1}{2}}\\
A_{2n-1}^{(2)} & u^{\frac{1}{2}} & -u^{\frac{1}{2}} & q^{\frac{1}{2}} &
-q^{\frac{1}{2}}\\
A_{2n-1}^{(2,\dagger)} & 1 & -1 & q^{\frac{1}{2}}u^{\frac{1}{2}} &
-q^{\frac{1}{2}}u^{\frac{1}{2}}\\
D_{n+1}^{(2)} & u & -1 & uq^{\frac{1}{2}} & -q^{\frac{1}{2}}\\
A_{2n}^{(2)} & u & -1 & u^{\frac{1}{2}}q^{\frac{1}{2}} & -u^{\frac{1}{2}}q^{\frac{1}{2}}\\\hline
\end{array}
\]
\tablecaption{Specializations of the Koornwinder polynomials yielding Macdonald
polynomials. Permutations of the parameters $(a,b,c,d)$ are allowed. The usual equal parameter
Macdonald polynomials are obtained for $u=t$.}
\label{table_spec}%
\end{table}

Observe that the Macdonald polynomials we consider are associated to non-simply laced affine root systems and we can consider the corresponding Macdonald
polynomials with unequal parameters $(t,u)$, with $t$ being associated to the roots of square length $2$, and $u$ being associated to roots of square root length $1$ (for $B_{n}^{(1)}$, $B_{n}^{(1,\dagger)}$, $D_{n+1}^{(2)}$), or $4$ (for $A_{2n-1}^{(2)}$, $A_{2n-1}^{(2,\dagger)}$, $C_{n}^{(1)}$), or both (for $A_{2n}^{(2)}$). Note that  $A_{2}^{(2)}$ has no roots of square length $2$ and in this case the Macdonald polynomials do not involve the parameter $t$. On the other hand, $D_{n}^{(1)}$ has only roots of square length $2$ and in this case the Macdonald polynomials do not involve the parameter $u$. To streamline the notation we will keep both parameters $t,u$ in the notation even in the situations when the particular Macdonald polynomials depend only one of them.

The more typically used equal parameter Macdonald
polynomials $P_{\lambda}^{T_{N}^{(a)}}(x;q,t)$ are obtained by setting $u=t$, that
is $P_{\lambda}^{T_{N}^{(a)}}(x;q,t)=P_{\lambda}^{T_{N}^{(a)}}(x; q,t,t)$. The parameter $u$ will allow us to consider the Hall-Littlewood
polynomials with unequal parameters which will be considered starting with Section~\ref{Sec_D_(n+1)^2}.\ 

\subsection{Hall-Littlewood polynomials and the \texorpdfstring{$q=0$}{q=0} Macdonald
specialization}

\label{subsec_HL}The Hall-Littlewood polynomials can be regarded as the $q=0$
specializations of  Macdonald polynomials. Observe first that by
Table~\ref{table_spec}, once a Macdonald specialization of type $T_{N}^{(a)}$
is performed in a Koornwinder polynomial, the additional $q=0$ specialization
only depends on the parabolic subsystem $T_{n}$ of $T_{N}^{(a)}$  obtained by
removing its $0$-node. For any dominant weight $\gamma\in P_{+}^{T_{n}}$, we
can thus define the Hall-Littlewood polynomial $P_{\gamma}^{T_{n}}(x;t,u)$ by
\[
P_{\gamma}^{T_{n}}(x; t,u)=P_{\gamma}^{T_{N}^{(a)}}(x; 0,t, u)\,,
\]
where $T_{N}^{(a)}$ is any affine root system with underlying classical root
system $T_{n}$. 

The previous Hall-Littlewood polynomials have a simpler definition independent of
Koornwinder-Macdonald polynomial theory. For each classical root system
$T_{n}=B_{n}$, $C_{n}$, or $D_{n}$, one can indeed establish that the family of
Hall-Littlewood polynomials $\{P_{\gamma}^{T_{n}}(x;u,t)\mid\gamma\in
P_{+}^{T_{n}}\}$ yields a basis of the character ring of type $T_{n}$ such
that%
\[
s_{\nu}^{T_{n}}(x)=\sum_{\gamma}K_{\nu,\gamma}(u,t)\,P_{\gamma}^{T_{n}%
}(x;u,t)\text{ for any }\nu\in P_{+}^{T_{n}}\,,%
\]
where the polynomials $K_{\nu,\gamma}(u,t)$ are $(u,t)$-deformations of the
generalized Kostka numbers $K_{\nu,\gamma}$ (see (\ref{Kostka})). One can show
(see \cite{NR}) that they can be obtained from the $(u,t)$-Kostant partition
function $\mathcal{P}_{u,t}^{T_{n}}$ defined from the series expansion%
\begin{align*}
\prod_{i=1}^{n}\frac{1}{1-ux_{i}}\prod_{1\leq i<j\leq n}\frac{1}{\left(
1-t\frac{x_{i}}{x_{j}}\right)  \left(  1-tx_{i}x_{j}\right)  }  &
=\sum_{\beta\in\mathbb{Z}^{n}}\mathcal{P}_{u,t}^{B_{n}}(\beta)x^{\beta}\text{
in type }B_{n},\\
\prod_{i=1}^{n}\frac{1}{1-ux_{i}^{2}}\prod_{1\leq i<j\leq n}\frac{1}{\left(
1-t\frac{x_{i}}{x_{j}}\right)  \left(  1-tx_{i}x_{j}\right)  }  &
=\sum_{\beta\in\mathbb{Z}^{n}}\mathcal{P}_{u,t}^{C_{n}}(\beta)x^{\beta}\text{
in type }C_{n},\\
\prod_{1\leq i<j\leq n}\frac{1}{\left(  1-t\frac{x_{i}}{x_{j}}\right)  \left(
1-tx_{i}x_{j}\right)  }  &  =\sum_{\beta\in\mathbb{Z}^{n}}\mathcal{P}%
_{u,t}^{D_{n}}(\beta)x^{\beta}\text{ in type }D_{n}.
\end{align*}
For any pair of dominant weights $\nu,\gamma$ in $P_{+}^{T_{n}},$ we then have
\begin{equation}
K_{\nu,\gamma}^{T_{n}}(u,t)=\sum_{w\in W^{T_{n}}}(-1)^{\ell(w)}\mathcal{P}%
_{u,t}^{T_{n}}\left(  w(\nu+\rho^{T_{n}})-(\gamma+\rho^{T_{n}})\right).
\label{DefKF(u,t)}%
\end{equation}
Observe that such double deformations of the generalized Kostka numbers have
been already introduced and studied in \cite{CKL2024} and \cite{Lecouvey2026}.\ We will
call them unequal-parameter Kostka-Foulkes polynomials. We will write for simplicity
$K_{\nu,\gamma}^{T_{n}}(t)=K_{\nu,\gamma}^{T_{m}}(t,t).\ $These last
$t$-deformations of the generalized Kostka numbers are also called Lusztig
$t$-analogues of weight multiplicities in the literature. They are know to
admit nonnegative integer coefficients. Similarly, the polynomials $P_{\gamma
}^{T_{n}}(x;t)=P_{\gamma}^{T_{n}}(x;t,t)$ are the usual (one-parameter)
Hall-Littlewood polynomials of type $T_{n}$.

\begin{remark}
\label{Rq_flip}  \ 

\begin{enumerate}
\item One can prove (see \cite{NR}) that the unequal-parameter Hall-Littlewood polynomials
of type $B_{n}$ also satisfy%
\[
P_{\nu}^{B_{n}}(x;u,t)=\frac{1}{W_{\nu}(u,t)}\frac{\left(  \sum_{w}%
(-1)^{\ell(w)}w\prod_{i=1}^{n}(1-ux_{i}^{-1})\prod_{1\leq i<j\leq n}%
(1-tx_{i}^{-1}x_{j}^{-1})x^{\nu+\rho^{B_{n}}}\right)  }{a_{\rho}^{B_n}}\,,%
\]
where
\[
W_{\nu}(u,t)=\sum_{w\in W_{\nu}}\prod_{\alpha\in I(w)}v^{\alpha}%
\]
with $I(w)=\{\alpha\in R_{+}^{B_{n}}\mid w(\alpha)\in-R_{+}^{B_{n}}\}$ and
$v^{\alpha}=t$ (resp. $v^{\alpha}=u$) if $\alpha$ is a long (resp. short)
root. There are similar formulas in type $C_{n}$ and $D_{n}$.

\item It is proved in \cite{CKL2024} that the polynomials $K_{\nu,\gamma}^{B_{n}}(u,t)$ have nonnegative integer coefficients when $\nu$ and $\gamma$ are
half-integer dominant weights. This is not true in general when $\nu$ and
$\gamma$ are partitions or in type $C_{n}$.

\item In the following sections, we will often need to consider the previous
Hall-Littlewood polynomials but for the root system of rank $m$ and the associated
character ring in the indeterminates $y=(y_{1},\ldots,y_{m})$ with the
deformation parameter $q$ instead of $t$. This swap will be a consequence of
the dual Cauchy formula for Koornwinder polynomials (\ref{Cauchy_Koorn}).
\end{enumerate}
\end{remark}

As explained in Remark \ref{Rq_flip},  in the forthcoming sections we will need
to swap the parameters $q$ and $t$ and also the ranks $n$ and $m$ in the
Koornwinder specializations; and next to specialize $t=0$ to get unequal parameter
Hall-Littlewood polynomials in $(u,t)$. Table~\ref{table_spec_swap} contains the mentioned specializations.

\begin{table}[ht]
\[
\begin{array}
[c]{@{}l@{\hskip20pt}l@{\hskip20pt}l@{\hskip20pt}l@{\hskip20pt}l}\hline
\text{Macdonald and classical types} & a & b & c & d\\\hline
D_{m}^{(1)},D_{m} & 1 & -1 & 0 & 0\\
B_{m}^{(1)},B_{m} & u & -1 & 0 & 0\\
B_{m}^{(1,\dagger)},D_{m} & 1 & -1 & 0 & 0\\
C_{m}^{(1)},C_{m} & u^{\frac{1}{2}} & -u^{\frac{1}{2}} & 0 & 0\\
A_{2m-1}^{(2)},C_{m} & u^{\frac{1}{2}} & -u^{\frac{1}{2}} & 0 & 0\\
A_{2m-1}^{(2,\dagger)},D_{m} & 1 & -1 & 0 & 0\\
D_{m+1}^{(2)},B_{m} & u & -1 & 0 & 0\\
A_{2m}^{(2)},B_{m} & u & -1 & 0 & 0\\\hline
\end{array}
\]
\tablecaption{Specializations of the Koornwinder polynomials yielding $(u,q)$ 
Hall-Littlewood polynomials in rank $m$ after swapping the ranks $n$ and $m$ 
and also the indeterminates $t$ and $q$ and putting $t=0$.}
\label{table_spec_swap}
\end{table}

\subsection{Weyl module characters and the \texorpdfstring{$t=0$}{t=0} Macdonald specialization}

Fix an affine root system of type $T_{N}^{(a)}$ with underlying classical root
system of type $T_{n}$. The goal of this paragraph is to establish the
following theorem which is crucial for our purposes.

\begin{theorem}
\label{Cor_PXS} For any dominant weight $\mu$ of type $T_{n}$, we have
\[
P_{\mu}^{T_{N}^{(a)}}(x;q,0)=\sum_{\lambda}X_{\lambda,\mu}^{T_{N}^{(a)}}(q)\,s_{\lambda}^{T_{n}}(x).
\]

\end{theorem}

In fact this theorem can be obtained by combining various results already appearing in the literature. We are going to proceed in three steps.

\subsubsection{Step 1: the affine root system $T_{N}^{(a)}$ is untwisted with
no dagger}

In this case our theorem is exactly Corollary~7.11 in \cite{LNSSS} (up to a slight
change of convention).

\subsubsection{Step 2: the affine root system $T_{N}^{(a)}$ is of twisted type
with no dagger}

Here we need a detour and consider some affine Demazure characters.\ Recall
first that the Demazure characters are the characters of modules associated
with the positive part $U_{\nu}^{+}(\widehat{\mathfrak{g}}\mathfrak{)}$ of
$U_{\nu}(\widehat{\mathfrak{g}}\mathfrak{)}$. For any dominant weight
$\Lambda$ and any element $w$ in $W_{\mathrm{a}}$, we have a Demazure $U_{\nu
}^{+}(\widehat{\mathfrak{g}}\mathfrak{)}$-module $V_{w}(\Lambda)$ contained in
$V(\Lambda)$ as a vector space. It also admits a crystal $B_{w}(\Lambda)$
contained in $B(\Lambda)$ as a subgraph, which thus inherits the underlying degree
$\mathrm{d}$. The general theory of Macdonald polynomials of simply laced
types ($A_{n}^{(1)}$ and $D_{n}^{(1)}$) and twisted affine types permits to
interpret their specialization at $t=0$ as certain Demazure characters of
irreducible highest weight modules associated with the affine root system
considered. Roughly speaking, this is done by setting $q=e^{-\delta}$, that is, 
by interpreting the dependence of the Demazure character on $\delta$ as a
$q$-grading.\ This crucial fact was proved in~\cite{Ion}, to which we refer for
a more complete presentation.\ 

\begin{theorem}
\label{Th_Bogdan}Assume $T_{N}^{(a)}$ is simply laced or of twisted type with
no dagger. For any dominant weight $\mu$ in $P_{+}$, there exists an affine
dominant weight $\Lambda$ and an element $w$ in $W_{\mathrm{a}}$ such that
\[
P_{\mu}^{T_{N}^{(a)}}(x;q,0)=\sum_{b\in B_{w}(\Lambda)}q^{\mathrm{d}%
(b)}x^{\mathrm{wt}(b)}.
\]
\end{theorem}

This interaction between the $t=0$ Macdonald specialization and the Demazure
characters theory also has a deeper interpretation in terms of crystal graphs.
In the cases we are considering here, there is indeed a strong connection
between a tensor product of column shape KR crystals and a certain Demazure crystal
inside a highest weight crystals of level $1$, that is, for dominant weights
$\Lambda$ which are affine fundamental weights.\ Here again, we only give
below a simplified version of a more precise theorem established in~\cite{ST}
(see Theorem~7.4). Before stating this theorem, we should mention that it only
holds in the cases when the considered column KR crystals are perfect.\ This
quite technical assumption (see Definition 2.4 in~\cite{ST}) is 
satisfied when $T_{N}^{(a)}$ is simply laced or of twisted type and, in
particular, even if $T_{N}^{(a)}$ is of type $B_{n}^{(1,\dag)}$, $A_{2n-1}%
^{(2,\dag)}$, or $A_{2n}^{(2,\dag)}$ (see \cite{FOS}).

\begin{theorem}
\label{Th_ST}Assume that $T_{N}^{(a)}$ is simply laced or of twisted
type.\ Then, each tensor product $B_{\mu}$ of perfect column KR crystals
admits a classical embedding\footnote{By a classical embedding, we mean a
graph embedding compatible with the classical crystal structure obtained by
removing the $0$-arrows.} in a level $1$ Demazure crystal which is also a
bijection on the associated sets of vertices. Moreover, the energy on $B$ can
be normalized in such a way it becomes equal to the grading $\mathrm{d}$
via this embedding.
\end{theorem}

We then derive our Theorem \ref{Cor_PXS} by combining Theorems \ref{Th_Bogdan}
and \ref{Th_ST}. As discussed above, for the simply laced untwisted types, this approach works  without using the results of \cite{LNSSS}.

\subsubsection{Step 3: the affine root system $T_{N}^{(a)}$ is of type
$B_{n}^{(1,\dag)}$, $A_{2n-1}^{(2,\dag)}$, or $A_{2n}^{(2,\dag)}$}

Here we proceed as in Step 2.\ As already observed, Theorem \ref{Th_ST} also
holds in these cases.\ The parameter specializations of the Koornwinder polynomials that 
correspond to these $T_{N}^{(a)}$ are not part of Theorem \ref{Th_Bogdan} as stated in \cite{Ion}.
Nevertheless, the technique of intertwiner operators for double affine Hecke algebras (on which Theorem \ref{Th_Bogdan} is based) is 
available for the full parameter Koornwinder polynomials (see, for example, \cite[\S2.6]{Ion2008}),  and, in particular, for the specializations relevant here. As in \cite{Ion}, the recursion given by the application of the intertwiner operators allows the identification of the $t=u=0$ limit of the relevant specialized Koornwinder polynomials with the graded character of the level 1 affine Demazure module specified by Theorem \ref{Th_ST}, for $T_{N}^{(a)}$ of the type considered here. This allows us to establish the validity of Theorem \ref{Th_Bogdan}, and consequently of Theorem \ref{Cor_PXS}, for  $T_{N}^{(a)}$ of type
$B_{n}^{(1,\dag)}$, $A_{2n-1}^{(2,\dag)}$, or $A_{2n}^{(2,\dag)}$.

\subsubsection{Connection with Weyl modules}

Theorem \ref{Th_Bogdan} connecting the $t=0$ specialization in Macdonald
polynomials with the affine Demazure characters does not hold for a general
untwisted non-simply laced affine root system. In fact the relevant general
context permitting to interpret these specializations as graded characters is
that of Weyl modules for current algebras.\ More precisely, in their study of
the category of finite-dimensional representations of quantum affine Lie
algebras $U_{\nu}(\widehat{\mathfrak{g}})$, Chari and Pressley
\cite{CP} defined some universal highest weight objects in this
category, called (local) Weyl modules. The Weyl modules are cyclic
indecomposable modules that play the role of standard objects in the category.
Any finite-dimensional cyclic indecomposable module is a quotient of a Weyl
module. The theory of Weyl modules transfers to the classical limit
$\nu\rightarrow1$, where it becomes the theory of Weyl modules for current
algebras; here, they play the role of standard objects in the category of
finite-dimensional \emph{graded} representations of the current algebra. The
grading of the current algebra representations arises from the action of the
scaling (sometimes called loop rotation) element in $\widehat{\mathfrak{g}}$
(see \S 2.1-2.3 in \cite{CIK}). At the level of characters, the grading is
thus captured by the imaginary root $\delta$: a $\widehat{\mathfrak{g}}%
$-character is seen as a graded $\mathfrak{g}$-character by denoting
$q=e^{\delta}$. If $T_{N}^{(a)}$ is twisted or simply laced untwisted, the
current algebra Weyl modules are precisely the symmetric (i.e. $\mathfrak{g}%
$-stable) Demazure modules of $V(\Lambda)$, for $\Lambda$ an affine dominant
weight of level one. Under the same constraint on $T_{N}^{(a)}$, the graded
characters of symmetric level-one Demazure modules were shown to be precisely
the $t=0$ specialization of the symmetric Macdonald polynomials of type
$T_{N}^{(a)}$~\cite{Ion}. What is true without any constraint on $T_{N}^{(a)}$
is that the $t=0$ specializations of the symmetric Macdonald polynomials of
type $T_{N}^{(a)}$ are the graded characters of the Weyl modules of the
current algebra of type $T_{N}^{(a)}$, see Theorem 4.2 in \cite{CIK}. A series of
results, culminating with the work of Lenart-Naito-Sagaki-Schilling-Shimozono~\cite{LNSSS}, shows that the Weyl modules have a graded crystal basis whose crystal is the tensor product of column KR-crystals; the energy function can be normalized so that it is identified with the grading.

\medskip

To facilitate the identification of Weyl module characters from Koornwinder specializations
in the forthcoming sections, we illustrate the specialization at $t=u=0$ in Table~\ref{table_spec_Demaz}.

\begin{table}[ht]
\[
\begin{array}
[c]{@{}l@{\hskip20pt}l@{\hskip20pt}l@{\hskip20pt}l@{\hskip20pt}l}\hline
\text{Macdonald type} & a & b & c & d\\\hline
D_{n}^{(1)} & 1 & -1 & q^{\frac{1}{2}} & -q^{\frac{1}{2}}\\
B_{n}^{(1)} & 0 & -1 & q^{\frac{1}{2}} & -q^{\frac{1}{2}}\\
B_{n}^{(1,\dagger)} & 1 & -1 & 0 & -q^{\frac{1}{2}}\\
C_{n}^{(1)} & 0 & 0 & 0 & 0\\
A_{2n-1}^{(2)} & 0 & 0 & q^{\frac{1}{2}} & -q^{\frac{1}{2}}\\
A_{2n-1}^{(2,\dagger)} & 1 & -1 & 0 & 0\\
D_{n+1}^{(2)} & 0 & -1 & 0 & -q^{\frac{1}{2}}\\
A_{2n}^{(2)} & 0 & -1 & 0 & 0\\\hline
\end{array}
\]
\tablecaption{Specializations of the Koornwinder polynomials yielding Weyl module characters.}
\label{table_spec_Demaz}%
\end{table}

\section{Dual Cauchy identity and the \texorpdfstring{$X=K$}{X=K} phenomenon in type  \texorpdfstring{$A_{n-1}^{(1)}$}{A{n-1}(1)}   } 
\label{sec_type_A}

In this Section, we reprove the equality between Kostka-Foulkes polynomials (in type
$A_{m-1})$ for a pair of dominant weights of level $n$ (i.e. a pair of
partitions contained in the rectangle $(n^{m})$) and 1-d sums corresponding to
tensor product of $m$ column KR crystals of affine type $A_{n-1}^{(1)}$.\ This
result was initially obtained by Nakayashiki and Yamada \cite{NY1997} in a purely
combinatorial way from the description of the Kostka-Foulkes polynomials in terms of
Lascoux-Sch\"{u}tzenberger's charge on semistandard tableaux \cite{LS1978} and that
of 1-d sums as generating functions for the energy statistic on finite affine
crystals. In contrast the proof we propose here is based on the dual Cauchy
formula for symmetric Macdonald polynomials. In particular, it will be simpler
to work with the ring of symmetric functions rather than in the character ring
of $\mathfrak{sl}_{n}$. We refer the reader to the classical book of Macdonald
\cite{mac} for more details on the notions introduced in this section.\ It is
written to be be read quite independently of the rest of the paper and we hope
it will help the reader to understand the main ideas which will be reinvested
in the study of the $X=K$ phenomenon beyond type $A$.

\subsection{Symmetric polynomials and Macdonald polynomials}

Let $\mathrm{Sym}[x_{1},\ldots,x_{n}]$ the ring of symmetric polynomials over
the rational functions in $\mathbb{Q}[q,t]$ where $q$ and $t$ are two
indeterminates.\ We denote by $\mathfrak{S}_{n}$ the symmetric group on the
set $\{1,\ldots,n\}$.\ It acts on $\mathbb{Z}^{n}=%
{\textstyle\bigoplus\limits_{i=1}^{n}}\mathbb{Z\varepsilon}_{i}$ by permutation of the coordinates. A partition
$\lambda$ of length at most $n$ is a sequence $\lambda=(\lambda_{1}\geq
\cdots\geq\lambda_{n}\geq0)$ and will be identified with its Young diagram.
The length of $\lambda$ is the number of its nonzero parts $\lambda_{i}$.\ We
denote by $\mathcal{P}_{n}$ the set of partitions with length at most $n$. The
orbits of the action of $\mathfrak{S}_{n}$ on $\mathbb{Z}^{n}$ are labelled by
the partitions of length at most $n$. For any such partition $\lambda$, we
define the monomial symmetric function by
\[
m_{\lambda}(x)=\sum_{\beta\in\mathfrak{S}_{n}\cdot\lambda}x^{\beta}\,,%
\]
where $\mathfrak{S}_{n}\cdot\lambda$ is the orbit of $\lambda$ under the
action of $\mathfrak{S}_{n}$ on $\mathbb{Z}^{n}$ (which so extends to
$\mathbb{Z}[x_{1},\ldots,x_{n}]$) and for any $\beta=(\beta_{1},\ldots
,\beta_{n})$ in $\mathbb{Z}^{n}$, we have $x^{\beta}=x_{1}^{\beta_{1}}\cdots
x_{n}^{\beta_{n}}$. The family $\{m_{\lambda}(x),\lambda\in\mathcal{P}_{n}\}$
is a basis of $\mathrm{Sym}[x_{1},\ldots,x_{n}]$.\ 

Now put $\partial=(n-1,\ldots,2,1)$. For any $\beta$ set%
\[
a_{\beta}(x)=\sum_{\sigma\in\mathfrak{S}_{n}}(-1)^{\ell(\sigma)}%
x^{\sigma(\beta)}\,,%
\]
where $\ell$ is the length function of $\mathfrak{S}_{n}$, that is the number
of elementary reflections $s_{i}=(i,i+1)$ appearing in any minimal length
decomposition of the permutation $\sigma$.\ For any $\lambda$ in
$\mathcal{P}_{n}$, define the Schur polynomial by%
\[
\mathsf{s}_{\lambda}(x)=\frac{a_{\lambda+\partial}(x)}{a_{\partial}(x)}.
\]
Then, the family $\{\mathsf{s}_{\lambda}(x),\lambda\in\mathcal{P}_{n}\}$ is
another basis of $\mathrm{Sym}[x_{1},\ldots,x_{n}]$. The Kostka numbers are
such that%
\begin{equation}
\mathsf{s}_{\lambda}(x)=\sum_{\mu\in\mathcal{P}_{n}}K_{\lambda,\mu}\,m_{\mu}(x). \label{SinM}
\end{equation}
The Kostka number $K_{\lambda,\mu}$ is in fact a nonnegative integer equal to
the dimension of the weight space of weight $\mu$ in the finite-dimensional
irreducible representation of the Lie algebra $\mathfrak{gl}_{n}(\mathbb{C)}$
with highest weight $\lambda$ (see \cite{FH} for an introduction on the
representation theory of Lie algebras). We can endow $\mathrm{Sym}%
[x_{1},\ldots,x_{n}]$ with a pairing $\langle\cdot,\cdot\rangle$ such that
$\langle\mathsf{s}_{\lambda}(x),\mathsf{s}_{\mu}(x)\rangle_{n}=\delta
_{\lambda,\mu}$ for any pair of partitions $(\lambda,\mu)$ in $\mathcal{P}%
_{n}^{2}$.\ 

The Hall-Littlewood polynomials can be regarded as $t$-interpolations between
the monomials and the Schur polynomials. They are defined by
\[
P_{\lambda}(x;t)=\frac{1}{\mathfrak{S}_{\lambda}(t)}\frac{\sum_{\sigma
\in\mathfrak{S}_{n}}(-1)^{\ell(w)}\sigma\left(  \prod_{1\leq i<j\leq
n}(1-t\frac{x_{j}}{x_{i}})x^{\lambda+\partial}\right)  }{a_{\partial}}\,,
\]
where $\mathfrak{S}_{\lambda}(t)=\sum_{\sigma\in\mathfrak{S}_{n}}%
t^{\ell(\sigma)}.$ We thus have $P_{\lambda}(x,0)=\mathsf{s}_{\lambda}(x)$ and
$P_{\lambda}(x,1)=m_{\lambda}(x)$ and in particular $K_{\lambda,\mu
}(1)=K_{\lambda,\mu}$ for any $\lambda,\mu$ in $\mathcal{P}_{n}$. Although
this is not obvious on their definition, they are indeed polynomials and their
coefficients belong to $\mathbb{Z}[t].$ They also yield a basis $\{P_{\lambda
}(x;t),\lambda\in\mathcal{P}_{n}\}$ of $\mathrm{Sym}[x_{1},\ldots,x_{n}]$.
This permits to define the $t$-Kostka-Foulkes polynomials by setting%
\begin{equation}
\mathsf{s}_{\lambda}(x)=\sum_{\mu\in\mathcal{P}_{n}}K_{\lambda,\mu
}(t)P_{\lambda}(x;t). \label{SinP}%
\end{equation}
Observe that both decompositions (\ref{SinM}) and (\ref{SinP}) are
unitriangular for the dominant order on $\mathcal{P}_{n}$. This means that we
have $K_{\lambda,\lambda}(t)=K_{\lambda,\lambda}=1$ and $K_{\lambda,\mu}(t)=0$
unless $\lambda\geq\mu,$ that is unless $\lambda-\mu$ in a nonnegative integer
combination of the simple roots $\alpha_{i}=\varepsilon_{i}-\varepsilon_{i+1}$
corresponding to the root system of type $A_{n-1}$. Since both families of
Schur and Hall-Littlewood polynomials have coefficients in $\mathbb{Z}[t]$, the Kostka
polynomials also belong to $\mathbb{Z}[t]$.\ In fact they have nonnegative
coefficients and we will see that among the many ways to prove this
fundamental result, one of them is to equate each Kostka-Foulkes polynomial with a 1-d
sum having by definition nonnegative integer coefficients as the generating
series of some particular vertices in affine crystals for the energy statistic.

The definition of the Macdonald polynomials is more involved.\ They can be
regarded as $q$-deformations $P_{\lambda}(x;q,t)$ of the Hall-Littlewood polynomials
$P_{\lambda}(x;t)$ with $\lambda$ in $\mathcal{P}_{n}$. In fact they are
defined as the unique basis of $\mathrm{Sym}[x_{1},\ldots,x_{n}]$
unitriangular on the monomial basis $\{m_{\lambda}(x),\lambda\in
\mathcal{P}_{n}\}$ and satisfying the orthogonality condition $\langle
P_{\lambda}(x;q,t),P_{\mu}(x;q,t)\rangle_{q,t}=\delta_{\lambda,\mu}$ where
$\langle\cdot,\cdot\rangle_{q,t}$ is a $(q,t)$-deformation of the previous
scalar product $\langle\cdot,\cdot\rangle$ making orthonormal the basis of
Schur functions. We do not really need the very definition of the symmetric
Macdonald polynomials in what follows but rather some of their crucial
properties. In particular, for any partition $\lambda$ in
$\mathcal{P}_{n}$ we have%
\[
P_{\lambda}(x,0,t)=P_{\lambda}(x;t)\,,
\]
that is, the Hall-Littlewood polynomials are the $q=0$ specializations of the Macdonald
polynomials. 

\subsection{The dual Cauchy identity for Hall-Littlewood polynomials}

Observe that for any partition $\lambda$ in the rectangle $(m^{n})$, its
conjugate partition is in the rectangle $(n^{m})$. We recall the theorem by
Nakayashiki and Yamada.

\begin{theorem}\label{thm_An-11}
\label{XvsK_type_A} For all $\lambda,\mu$ two partitions in the rectangle
$(m^{n})$, we have%
\[
X_{\lambda,\mu}(q)=K_{\lambda^{\prime},\mu^{\prime}}(q)\,,
\]
where $X_{\lambda,\mu}(q)$ is the column 1-d sum of type $A_{n-1}^{(1)}$
determined by the columns of the partition $\mu$ (see Definition \ref{Def1-d})
and $K_{\lambda^{\prime},\mu^{\prime}}(q)$ the Kostka-Foulkes polynomial of type
$A_{m-1}$ associated with the pair of partitions $(\lambda^{\prime}%
,\mu^{\prime})$.
\end{theorem}

The dual Cauchy identity can be written
\begin{equation}
\prod_{\substack{1\leq i\leq n\\1\leq j\leq m}}(1+x_{i}y_{j})=\sum
_{\lambda\subseteq(m^{n})}\mathsf{s}_{\lambda}(x)\,\mathsf{s}_{\lambda^{\prime}%
}(y). \label{Cauchy_A}%
\end{equation}
Then, as explained in \cite[Chapter VI, $(2.7)$]{mac}, any pair of bases
$\left(  \,(u_{\lambda}(x))_{\lambda\subseteq(m^{n})}\,,\,(v_{\lambda^{\prime
}})_{\lambda\subseteq(m^{n})}\,\right)  $ verifying
\begin{equation}
\left\langle u_{\lambda},v_{\mu^{\prime}}\right\rangle =\delta_{\lambda,\mu
}\text{ for any }(\lambda,\mu)\ \text{in }(m^{n})\times(m^{n})
\label{dual_bases_A}%
\end{equation}
yields a Cauchy-type identity%

\begin{equation}
\prod_{\substack{1\leq i\leq n\\1\leq j\leq m}}(1+x_{i}y_{j})=\sum
_{\lambda\subseteq(m^{n})}u_{\lambda}(x)\,v_{\lambda^{\prime}}(y)
\label{Cauchy_A_gen}\,,%
\end{equation}
and conversely. We can thus introduce the basis $\{\mathsf{Q}_{\lambda
}(x;q)\mid\lambda\subseteq(m^{n})\}$ such that
\begin{equation}
\left\langle \mathsf{Q}_{\lambda}(x;q),P_{\mu^{\prime}}(y,q\right\rangle
=\delta_{\lambda,\mu}\text{ for any }(\lambda,\mu)\subseteq(m^{n})\times(m^{n})
\label{DefQ'}%
\end{equation}
and get the dual Cauchy identity.

\begin{equation}
\prod_{\substack{1\leq i\leq n\\1\leq j\leq m}}(1+x_{i}y_{j})=\sum
_{\mu\subseteq(m^{n})}\mathsf{Q}_{\mu}(x;q)P_{{\mu^{\prime}}}(y,q).
\label{Cauchy_A_HL}
\end{equation}
One can observe here that, with the notation of Macdonald's book, we have
$\mathsf{Q}_{\mu}(x;q)=\omega\left(  Q_{\mu^{\prime}}(x;q)\right)  $, i.e. the
polynomial $\mathsf{Q}_{\mu}(x;q)$ is the just the image of the modified
Hall-Littlewood polynomial $Q_{\mu}^{\prime}(x;q)$ under the involution
$\omega$ in the ring of symmetric functions.

\begin{lemma}
\label{Lemma_Qins}We have for any partition $\mu\subseteq(m^{n})$%
\[
\mathsf{Q}_{\mu}(x;q)=\sum_{\lambda\subseteq(m^{n})}K_{\lambda^{\prime}%
,\mu^{\prime}}(q)s_{\lambda}(x).
\]

\end{lemma}

\begin{proof}
Let us set%
\[
\mathsf{Q}_{\mu}(x,q)=\sum_{\lambda\subseteq(m^{n})}a_{\lambda,\mu}s_{\lambda
}(x).
\]
Then we can write%
\[
a_{\lambda,\mu}=\langle\mathsf{Q}_{\mu}(x;q),s_{\lambda^{\prime}}%
(y)\rangle=\sum_{\nu^{\prime}}\langle\mathsf{Q}_{\mu}(x;q),P_{\nu^{\prime}%
}(x;q)\rangle K_{\lambda^{\prime},\nu^{\prime}}(q)=K_{\lambda^{\prime}%
,\mu^{\prime}}(q)\,,
\]
where the first and last equalities follow from (\ref{SinP}) and
(\ref{DefQ'}), respectively.
\end{proof}

\subsection{The dual Cauchy identity for Macdonald polynomials}

In \cite[VI, $(5.4)$ p. 329]{mac}, the following dual Cauchy identity for the
Macdonald polynomials is established:%

\[
\prod_{\substack{1\leq i\leq n\\1\leq j\leq m}}(1+x_{i}y_{j})=\sum
_{\lambda\subseteq(m^{n})}P_{\lambda}(x;q,t)\,P_{\lambda^{\prime}}(y;t,q).
\]
Now, we let $t$ tends to $0$ in the above expression. According to Theorem
\ref{Th_Bogdan}, the polynomial $P_{\lambda}(x;q,t)$ on the left specializes
to the Demazure character $P_{\lambda}(x,q,0)$, whereas the polynomial on the
right specializes to the Hall-Littlewood $P_{\lambda^{\prime}}(y,q)$. We thus obtain%

\[
\prod_{\substack{1\leq i\leq n\\1\leq j\leq m}}(1+x_{i}y_{j})=\sum
_{\lambda\subseteq(m^{n})}P_{\lambda}(x;q,0)\,P_{\lambda^{\prime}}(y;q).
\]
Comparing with (\ref{Cauchy_A_HL}), this yields
\[
P_{\lambda}(x;q,0)=\mathsf{Q}_{\lambda}(x;q).
\]
We can now use Lemma \ref{Lemma_Qins} and Theorem \ref{Cor_PXS} to get the
$X=K$ equality of Theorem \ref{XvsK_type_A}, namely%
\[
X_{\lambda,\mu}(q)=K_{\lambda^{\prime},\mu^{\prime}}(q)\text{ for any
}(\lambda,\mu)\text{ in }(m^{n})\times(m^{n}).
\]

\section{Dual Cauchy identity and the \texorpdfstring{$X=K$}{X=K} phenomenon in type \texorpdfstring{$A_{2n-1}^{(2)}$}{A{2n-1}(2)}}
\label{sec_type_C}

\label{Sec_A_(2n-1)^(2)}In this section, we prove that the 1-d sums of level
$m$ and type $A_{2n-1}^{(2)}$ coincide with the Lusztig $q$-analogues of type
$C_{m}$ indexed by pairs of partitions in the rectangle $(n^{m})$.

\subsection{Duality and main theorem}

Let $\lambda,\mu\subseteq(m^{n})$. Recall the following notation.

\begin{itemize}
\item $X_{\lambda,\mu}^{A_{2n-1}^{(2)}}(q)$ is the one-dimensional sum associated
with the dominant weight $\lambda$ of the affine crystal $B^{(\mu_{1}^{\prime
},1)}\otimes\cdots\otimes B^{(\mu_{m}^{\prime},1)}$. Here $B^{(k,1)}$ denotes
the KR crystal of type $A_{2n-1}^{(2)}$ and column shape of height $k$. By
removing the $0$-arrows in $B^{(k,1)}$, we get a type $C_{n}$-crystal
isomorphic to the sum%
\[
B(\omega_{k})\oplus B(\omega_{k-2})\oplus\cdots\oplus B(\omega
_{k\operatorname{mod}2})
\]
where $B(\omega_{0})=B(0)$ is the crystal of the trivial representation (one
vertex with no arrow).

\item $K_{\widehat{\lambda},\widehat{\mu}}^{C_{m}}(q)$ is the Kostka-Foulkes polynomial
of type $C_{m}$ associated with the partitions $\widehat{\lambda}$ and $\widehat{\mu}$.
\end{itemize}

The goal of this section is to establish the following theorem.

\begin{theorem}\label{thm_A2n-12}
\label{XvsK_type_C} For all $\lambda,\mu\subseteq(m^{n})$, we have
\[
X_{\lambda,\mu}^{A_{2n-1}^{(2)}}(q)=K_{\widehat{\lambda},\widehat{\mu}}^{C_{m}}(q)\,.
\]

\end{theorem}

This will be done by replacing the dual Cauchy identity on Macdonald
polynomial by a relevant specialization in the dual Cauchy identity for
Koornwinder polynomials. When $q=1$, observe this is just the usual
$C_{n}\times C_{m}$ Howe duality between multiplicities in tensor product of
$k$-wedges product of $\mathbb{C}^{n}$ and weight multiplicities in
irreducible representations of $\mathfrak{sp}_{2m}$, see \cite{H95}.

\begin{remark}\label{RemGGL}
    As explained in \cite{GGL}, there is a simple duality between King tableaux of type $C_m$ and tensor products of $m$ type $C_m$ columns given highest weight vertices. From Theorem \ref{XvsK_type_C}, it becomes then natural to define the symplectic charge of a King tableau as the energy of its associated tensor product of columns. This yields an analogue of Lascoux-Sch\"utzenberger's description \cite{LS1978} of the usual Kostka-Foulkes polynomials in terms of semistandard tableaux.
\end{remark}

\subsection{The symplectic dual Cauchy identity}

Recall that $x=(x_{1},\ldots,x_{n})$ and $y=(y_{1},\ldots,y_{m})$ are two sets
of indeterminates. Also $s_{\lambda}^{C_{n}}(x)$ is the Weyl character of type
$C_{n}$ associated with the partition $\lambda\subseteq(m^{n})$. There exists
a type $C$ analogue of the (dual) Cauchy identity, that can be found in
\cite{King2023}, namely%
\begin{equation}
\prod_{\substack{1\leq i\leq n\\1\leq j\leq m}}(x_{i}+x_{i}^{-1}+y_{j}%
+y_{j}^{-1})=\sum_{\lambda\subseteq(m^{n})}s_{\lambda}^{C_{n}}(x)\,s_{\widehat
{\lambda}}^{C_{m}}(y) \label{Cauchy_C}\,.%
\end{equation}
Let $\mathrm{char}_{\leq m}^{C_{n}}(x)$ and $\mathrm{char}_{\leq n}^{C_{m}%
}(y)$ be the subspaces of the character ring of type $C_{n}$ and $C_{m}$ with
basis $\{s_{\lambda}^{C_{n}}(x)\mid\lambda\subseteq(m^{n})\}$ and $\{s_{\widehat
{\lambda}}^{C_{m}}(y)\mid\lambda\subseteq(m^{n})\}$, respectively. Denote by
$\left\langle .,.\right\rangle _{C_{n}\times C_{m}}$ the pairing on
$\mathrm{char}_{\leq m}^{C_{n}}(x)\times\mathrm{char}_{\leq n}^{C_{m}}(y)$
such that $\langle s_{\lambda}^{C_{n}},s_{\widehat{\mu}}^{C_{m}}\rangle_{C_{n}\times C_{m}}=\delta_{\lambda,\mu}$ for all $\lambda,\mu\subseteq(m^{n})$.
Similarly to (\ref{dual_bases_A}), any pair of bases
\[
\left(  \,(u_{\lambda})_{\lambda\subseteq(m^{n})}\quad,\quad(v_{\widehat{\lambda}%
})_{\lambda\subseteq(m^{n})}\,\right)
\]
verifying

\begin{equation}
\left\langle u_{\lambda},v_{\widehat{\mu}}\right\rangle _{C_{n}\times C_{m}%
}=\delta_{\lambda,\mu} \label{dual_bases}%
\end{equation}
yields a Cauchy-type identity%

\begin{equation}
\prod_{\substack{1\leq i\leq n\\1\leq j\leq m}}(x_{i}+x_{i}^{-1}+y_{j}%
+y_{j}^{-1})=\sum_{\lambda\subseteq(m^{n})}u_{\lambda}(x)\,v_{\widehat{\lambda}}(y).
\label{Cauchy_C_gen}%
\end{equation}

\subsection{Symplectic dual Cauchy identity for Hall-Littlewood polynomials}

Let $\{\mathsf{Q}_{\mu}^{C_{n}}(x;q),\mu\subseteq(m^{n})\}$ be the dual basis of
the Hall-Littlewood basis $\{P_{\widehat{\mu}}^{C_{m}}(y;q)\mid\mu\subseteq
(m^{n})\}$ for the previous $\left\langle .,.\right\rangle _{C_{n}\times
C_{m}}$-pairing. By the previous arguments, we have%
\begin{equation}
\prod_{\substack{1\leq i\leq n\\1\leq j\leq m}}(x_{i}+x_{i}^{-1}+y_{j}%
+y_{j}^{-1})=\sum_{\mu\subseteq(m^{n})}\mathsf{Q}_{\mu}^{C_{n}}(x;q)P_{\widehat
{\mu}}^{C_{m}}(y;q). \label{Cauchy_C_HL}%
\end{equation}
By arguing as in Lemma \ref{Lemma_Qins}, we get%
\begin{equation}
\mathsf{Q}_{\mu}^{C_{n}}(y;q)=\sum_{\lambda\subseteq(m^{n})}K_{\widehat{\lambda
},\widehat{\mu}}^{C_{m}}(q)\,s_{\lambda}^{C_{m}}(y)\,, \label{QinsC}%
\end{equation}
since we have the identity%
\[
s_{\widehat{\lambda}}^{C_{m}}(y)=\sum_{\mu\subseteq(m^{n})}K_{\widehat{\lambda}%
,\widehat{\mu}}^{C_{m}}(q)\,P_{\widehat{\mu}}^{C_{m}}(y).
\]
Now, in view of proving \Cref{XvsK_type_C}, we specialize the parameters in
order to obtain a type $A_{2n-1}^{(2)}$ Macdonald polynomial as the left
polynomial in (\ref{Cauchy_Koorn}). Recall that $P_{\mu}^{A_{2n-1}^{(2)}%
}(x;q,t)$ is the Macdonald polynomial of type $A_{2n-1}^{(2)}$ (with $u=t$),
we obtain the following relation by \Cref{table_spec}:%
\[
\prod_{\substack{1\leq i\leq n\\1\leq j\leq m}}(x_{i}+x_{i}^{-1}-y_{j}%
-y_{j}^{-1})=\sum_{\mu\subseteq(m^{n})}(-1)^{|\widehat{\mu}|}P_{\mu}%
^{A_{2n-1}^{(2)}}(x;q,t)\,P_{\widehat{\mu}}(y;t^{\frac{1}{2}},-t^{\frac{1}{2}%
},q^{\frac{1}{2}},-q^{\frac{1}{2}};t,q).
\]
Now, we will let $t\rightarrow0$ in the above expression. On the one hand,
\[
\lim_{t\rightarrow0}P_{\mu}^{A_{2n-1}^{(2)}}(x;q,t)=P_{\mu}^{A_{2n-1}^{(2)}%
}(x;q,0)
\]
is a Demazure character of type $A_{2n-1}^{(2)}$ by Theorem \ref{Th_Bogdan}.
On the other hand, note that the polynomial on the right $P_{\widehat{\lambda}%
}(y;t^{\frac{1}{2}},-t^{\frac{1}{2}},q^{\frac{1}{2}},-q^{\frac{1}{2}};t,q)$
does not quite yield a Macdonald polynomial, but we have nevertheless
\[%
\begin{array}
[c]{rcll}%
\lim_{t\rightarrow0}P_{\widehat{\mu}}(y;t^{\frac{1}{2}},-t^{\frac{1}{2}}%
,q^{\frac{1}{2}},-q^{\frac{1}{2}};t,q) & = & \lim_{t\rightarrow0}P_{\widehat{\mu}%
}(y;q^{\frac{1}{2}},-q^{\frac{1}{2}},t^{\frac{1}{2}},-t^{\frac{1}{2}};t,q) &
\begin{array}
[t]{l}%
\text{by permuting}\\
\text{the parameters}%
\end{array}
\\
& = & P_{\widehat{\mu}}(y;q^{\frac{1}{2}},-q^{\frac{1}{2}},0,0;0,q) & \\
& = & P_{\widehat{\mu}}^{A_{2m-1}^{(2)}}(y;0,q) &
\begin{array}
[t]{l}%
\text{by \Cref{table_spec}}\\
\text{with $t\leftrightarrow q$}%
\end{array}
\text{ }\\
& = & P_{\widehat{\mu}}^{C_{m}}(y;q)\,, &
\end{array}
\]
where $P_{\widehat{\mu}}^{C_{m}}(y;q)$ is the Hall-Littlewood polynomial of type
$C_{m}.$ Therefore, taking the limit $t\rightarrow0$ in (\ref{Cauchy_C_HL})
yields the identity
\begin{equation}
\prod_{\substack{1\leq i\leq n\\1\leq j\leq m}}(x_{i}+x_{i}^{-1}-y_{j}%
-y_{j}^{-1})=\sum_{\mu\subseteq(m^{n})}(-1)^{|\widehat{\lambda}|}P_{\mu}%
^{A_{2n-1}^{(2)}}(x;q,0)\,P_{\widehat{\mu}}^{C_{m}}(y;q). \label{Cauchy_C_HL2}%
\end{equation}
Now, substituting $y\leftarrow-y$, we obtain%
\begin{equation}
\prod_{\substack{1\leq i\leq n\\1\leq j\leq m}}(x_{i}+x_{i}^{-1}+y_{j}%
+y_{j}^{-1})=\sum_{\lambda\subseteq(m^{n})}P_{\mu}^{A_{2n-1}^{(2)}%
}(x;q,0)\,P_{\widehat{\mu}}^{C_{m}}(y;q). \label{Cauchy_C_HL3}%
\end{equation}
Indeed, the transfer matrix between the basis of Hall-Littlewood polynomials
$P_{\widehat{\mu}}^{C_{m}}(y;q)$ and that of the Weyl characters $s_{\widehat{\lambda
}}^{C_{m}}(y)$ is $(K_{\widehat{\lambda},\widehat{\mu}}^{C_{m}}(q))^{-1}$, the inverse of the
matrix whose coefficients are the Lusztig $q$-analogues of type $C_{m}%
$.\ Since $K_{\widehat{\lambda},\widehat{\mu}}^{C_{m}}(q)\neq0$ only when $\left\vert
\widehat{\lambda}\right\vert =\left\vert \widehat{\mu}\right\vert \operatorname{mod}%
2$, the decomposition of each polynomial $P_{\widehat{\mu}}^{C_{m}}(y;q)$ in the
basis of Weyl characters makes appear nonzero coefficients only for the
$s_{\widehat{\lambda}}^{C_{m}}(y)$'s with $\left\vert \widehat{\lambda}\right\vert
=\left\vert \widehat{\mu}\right\vert \operatorname{mod}2$. Now, we have for any
such character $s_{\widehat{\lambda}}^{C_{m}}(-y)=(-1)^{\left\vert \widehat{\lambda
}\right\vert }s_{\widehat{\lambda}}^{C_{m}}(-y)$ and therefore also $P_{\widehat{\mu}%
}^{C_{m}}(-y;q)=(-1)^{\left\vert \widehat{\lambda}\right\vert }P_{\widehat{\mu}%
}^{C_{m}}(y;q)$.

\begin{remark}
Observe we will also have $s_{\widehat{\lambda}}^{D_{m}}(-y)=(-1)^{\left\vert
\widehat{\lambda}\right\vert }s_{\widehat{\lambda}}^{D_{m}}(-y)$ for the Weyl
characters of type $D_{m}$ parametrized by a partition but a similar
identities does not hold in type $B_{m}$.
\end{remark}

\medskip

Comparing (\ref{Cauchy_C_HL}) and (\ref{Cauchy_C_HL3}), we deduce the equality
$\mathsf{Q}_{\mu}^{C_{n}}(x;q)=P_{\mu}^{A_{2n-1}^{(2)}}(x;q,0)$ for any
$\mu\subseteq(m^{n})$.\ This concludes the proof of Theorem \ref{XvsK_type_C} by
using (\ref{QinsC}) and Theorem \ref{Cor_PXS}, since for any
$\mu\subseteq (m^{n})$ we have%
\[
\sum_{\lambda\subseteq(m^{n})}K_{\widehat{\lambda},\widehat{\mu}}^{C_{m}}%
(q)\,s_{\lambda}^{C_{m}}(x)=\mathsf{Q}_{\mu}^{C_{n}}(x;q)=P_{\mu}^{A_{2n-1}%
^{(2)}}(x;q,0)=\sum_{\lambda}X_{\lambda,\mu}^{A_{2n-1}^{(2)}}(q)\,s_{\lambda
}^{C_{m}}(x).
\]

\section{Dual Cauchy identity and the \texorpdfstring{$X=K$}{X=K} phenomenon in type \texorpdfstring{$A_{2n-1}^{(2,\dagger)}$}{A{2n-1}(2)}}
\label{type_D_int}

We now prove that the 1-d sums of level $m$ and type $A_{2n-1}^{(2,\dagger)}$
coincide with the Lusztig $q$-analogues of type $D_{m}$ indexed by pairs of
partitions in the rectangle $(n^{m})$. Let $\lambda,\mu\subseteq(m^{n})$. Recall the following notation.

\begin{itemize}
\item $X_{\lambda,\mu}^{A_{2n-1}^{(2,\dagger)}}(q)$ is the one-dimensional sum
associated with the dominant weight $\lambda$ of the Kirillov-Reshetikhin
crystal $B^{(\mu_{1}^{\prime},1)}\otimes\cdots\otimes B^{(\mu_{m}^{\prime}%
,1)}$. Here $B^{(k,1)}$ denotes the KR crystal of type $A_{2n-1}^{(2,\dagger
)}$ and column shape of height $k$. By removing the $0$-arrows in $B^{(k,1)}$,
we get a connected type $D_{n}$-crystal isomorphic to
\[
\left\{
\begin{array}
[c]{l}%
B(\omega_{k})\text{ if }0\leq k\leq n-2\\
B(\omega_{n}+\omega_{n-1})\text{ if }k=n-1\\
B(2\omega_{n})\text{ if }k=n.
\end{array}
\right.
\]

\item $K_{\widehat{\lambda},\widehat{\mu}}^{D_{m}}(q)$ is the Kostka-Foulkes polynomial
of type $D_{m}$ associated with the partitions $\widehat{\lambda}$ and $\widehat{\mu}$.
\end{itemize}

The goal of this section is to establish the following theorem.

\begin{theorem}\label{thm_A2n-12dag}
\label{XvsK_type_D} For all $\lambda,\mu\subseteq(m^{n})$, we have
\[
X_{\lambda,\mu}^{A_{2n-1}^{(2,\dagger)}}(q)=K_{\widehat{\lambda},\widehat{\mu}}%
^{D_{m}}(q).
\]

\end{theorem}

When $q=1$, observe this is just the usual $D_{n}\times D_{m}$ Howe duality
between multiplicities in tensor product of $k$-wedges product of
$\mathbb{C}^{n}$ and weight multiplicities in irreducible representations of
$\mathfrak{o}_{2m}$, see \cite{H95}.

The proof follows essentially the same line as for the equality $X_{\lambda
,\mu}^{A_{2n-1}^{(2)}}(q)=K_{\widehat{\lambda},\widehat{\mu}}^{C_{m}}(q)$ detailed in
the previous section.\ We have this time the dual Cauchy identity%
\begin{equation}
\prod_{\substack{1\leq i\leq n\\1\leq j\leq m}}(x_{i}+x_{i}^{-1}+y_{j}%
+y_{j}^{-1})=\sum_{\lambda\subseteq(m^{n})}s_{\lambda}^{D_{n}}(x)\,s_{\widehat
{\lambda}}^{D_{m}}(y) \label{Cauchy_D}\,,%
\end{equation}
for which we refer to Proposition 5\ in \cite{Golt}.\ Let $\mathrm{char}_{\leq
m}^{D_{n}}(x)$ and $\mathrm{char}_{\leq n}^{D_{m}}(y)$ be the subspaces of the
character ring of type $D_{n}$ and $D_{m}$ with basis $\{s_{\lambda}^{D_{n}%
}(x)\mid\lambda\subseteq (m^{n})\}$ and $\{s_{\widehat{\lambda}}^{D_{m}}(y)\mid
\lambda\subseteq (m^{n})\}$, respectively. Denote by $\left\langle
.,.\right\rangle _{D_{n}\times D_{m}}$ the pairing on $\mathrm{char}_{\leq
m}^{D_{n}}(x)\times\mathrm{char}_{\leq n}^{D_{m}}(y)$ such that $\langle
s_{\lambda}^{D_{n}},s_{\widehat{\mu}}^{D_{m}}\rangle_{D_{n}\times D_{m}}%
=\delta_{\lambda,\mu}$ for all $\lambda,\mu\subseteq(m^{n})$.

\medskip

Let $\{\mathsf{Q}_{\mu}^{D_{n}}(x,q),\mu\subseteq(m^{n})\}$ be the dual basis of
the Hall-Littlewood basis $\{P_{\widehat{\mu}}^{D_{m}}(y,q)\mid\mu\subseteq
(m^{n})\}$ for the previous $\left\langle .,.\right\rangle _{D_{n}\times
D_{m}}$-pairing. We get this time
\begin{equation}
\prod_{\substack{1\leq i\leq n\\1\leq j\leq m}}(x_{i}+x_{i}^{-1}+y_{j}+y_{j}^{-1})=\sum_{\mu\subseteq(m^{n})}\mathsf{Q}_{\mu}^{D_{n}}(x;q)\,P_{\widehat{\mu}}^{D_{m}}(y;q). \label{Cauchy_D_HL}
\end{equation}
and%
\begin{equation}
\mathsf{Q}_{\mu}^{D_{m}}(x,q)=\sum_{\lambda\subseteq(m^{n})}K_{\widehat{\lambda
},\widehat{\mu}}^{D_{m}}(q)\,s_{\lambda}^{D_{m}}. \label{QinsD}%
\end{equation}
We now specialize the parameters in order to obtain a type $A_{2n-1}%
^{(2,\dagger)}$ Macdonald polynomial as the left polynomial in
(\ref{Cauchy_Koorn}). Recalling that $P_{\mu}^{A_{2n-1}^{(2,\dagger)}}(x;q,t)$ is
the Macdonald polynomial of type $A_{2n-1}^{(2,\dagger)}$, we obtain%
\[
\prod_{\substack{1\leq i\leq n\\1\leq j\leq m}}(x_{i}+x_{i}^{-1}-y_{j}%
-y_{j}^{-1})=\sum_{\mu\subseteq(m^{n})}(-1)^{|\widehat{\lambda}|}P_{\mu}%
^{A_{2n-1}^{(2,\dagger)}}(x;q,t)\,P_{\widehat{\mu}}(y;1,-1,q^{\frac{1}{2}}%
t^{\frac{1}{2}},-q^{\frac{1}{2}}t^{\frac{1}{2}};t,q)\,,
\]
and we let $t$ tends to $0$. The $t=0$ limit%
\[
\lim_{t\rightarrow0}P_{\mu}^{A_{2n-1}^{(2,\dagger)}}(x;q,t)=P_{\mu}%
^{A_{2n-1}^{(2,\dagger)}}(x;q,0)
\]
is a Demazure character of type $A_{2n-1}^{(2,\dagger)}$ by Theorem
\ref{Th_Bogdan}. For the polynomial $P_{\widehat{\mu}}(y;1,-1,q^{\frac{1}{2}%
}t^{\frac{1}{2}},-q^{\frac{1}{2}}t^{\frac{1}{2}};t,q)$ we obtain%
\[
\lim_{t\rightarrow0}P_{\widehat{\mu}}(y;1,-1,q^{\frac{1}{2}}t^{\frac{1}{2}%
},-q^{\frac{1}{2}}t^{\frac{1}{2}};t,q)=P_{\widehat{\mu}}(y;1,-1,0,0;0,q)=P_{\widehat
{\mu}}^{D_{m}}(y;q).
\]
where $P_{\widehat{\mu}}^{D_{m}}(y;q)$ is the Hall-Littlewood polynomial of type
$D_{m}.$ Thus, we derive the identity
\[
\prod_{\substack{1\leq i\leq n\\1\leq j\leq m}}(x_{i}+x_{i}^{-1}-y_{j}%
-y_{j}^{-1})=\sum_{\mu\subseteq(m^{n})}(-1)^{|\widehat{\lambda}|}P_{\mu}%
^{A_{2n-1}^{(2,\dagger)}}(x;q,0)\,P_{\widehat{\mu}}^{D_{m}}(y;q)
\]
and, with the same argument as in type $C_{m}$, the substitution
$y\leftarrow-y$ permits to write%
\[
\prod_{\substack{1\leq i\leq n\\1\leq j\leq m}}(x_{i}+x_{i}^{-1}+y_{j}%
+y_{j}^{-1})=\sum_{\mu\subseteq (m^{n})}P_{\mu}^{A_{2n-1}^{(2,\dagger)}%
}(x;q,0)\,P_{\widehat{\mu}}^{D_{m}}(y;q).
\]
We then deduce the identity%
\[
\mathsf{Q}_{\mu}^{D_{m}}(x;q)=P_{\mu}^{A_{2n-1}^{(2,\dagger)}}(x;q,0)\text{
for any }\mu\subseteq (m^{n}).
\]
As in the type $A_{2n-1}^{(2)}$-case, this concludes the proof of Theorem
\ref{XvsK_type_D} by using (\ref{QinsD}) and Theorem \ref{Cor_PXS}.

\section{Double deformation of weight multiplicities and the \texorpdfstring{$X=K$}{X=K} phenomenon
in types \texorpdfstring{$A_{2n}^{(2)}$}{A2n(2)} and \texorpdfstring{$D_{n+1}^{(2)}$}{D{n+1}(2)}}
\label{sec_type_BD}

\label{Sec_D_(n+1)^2}

\subsection{Weyl characters of types \texorpdfstring{$B_{m}$}{Bm} and \texorpdfstring{$C_{m}$}{Cm}}

Let us compare in this paragraph the Weyl characters of types $B_{m}$ and
$C_{m}$ respectively for half-integers and integer dominant weights. Recall
first that
\[
(m,m-1,\ldots,1)=\rho_{C_{n}}=\rho_{B_{m}}+\frac{1}{2}(1,\ldots,1)=\rho
_{B_{m}}+\omega_{m}^{B_{m}}.
\]
Now consider a partition $\lambda$ with at most $m$ parts. The associated
character of type $C_{m}$ satisfies%
\[
s_{\lambda}^{C_{m}}=\frac{a_{\lambda+\rho_{m}^{C_{m}}}}{a_{\rho_{m}^{C_{m}}}}\,,%
\]
where for any $\beta\in\frac{1}{2}\mathbb{Z}^{m}$ we have
\[
a_{\beta}=\sum_{w\in W}(-1)^{\ell(w)}y^{w(\beta)}.
\]
Observe this definition is the same in type $B_{m}$ and $C_{m}$ because the
Weyl group is the same.\ Then, we can use the trick%
\[
a_{\beta+\rho_{m}^{C_{m}}}=a_{(\beta+\omega_{m}^{B_{m}})+\rho_{m}^{B_{m}}}.
\]
Once plugged in the previous WCF for $s_{\lambda}^{C_{m}}$, this gives%
\begin{equation}
s_{\lambda}^{C_{m}}=\frac{a_{(\lambda+\omega_{m}^{B_{m}})+\rho_{m}^{B_{m}}}%
}{a_{\omega_{m}^{B_{m}}+\rho_{m}^{B_{m}}}}=\frac{a_{(\lambda+\omega_{m}%
^{B_{m}})+\rho_{m}^{B_{m}}}}{a_{\rho_{m}^{B_{m}}}}\times\frac{a_{\rho
_{m}^{B_{m}}}}{a_{\omega_{m}^{B_{m}}+\rho_{m}^{B_{m}}}}=s_{\lambda+\omega
_{m}^{B_{m}}}^{B_{m}}\times\frac{1}{s_{\omega_{m}^{B_{m}}}^{B_{m}}}.
\label{identity}%
\end{equation}
Now recall that the highest weight representation of type $B_{m}$ of highest
weight $\omega_{m}^{B_{m}}$ is the spin representation with character%
\[
s_{\omega_{m}^{B_{m}}}^{B_{m}}=\prod_{j=1}^{m}(x_{j}^{1/2}+x_{j}^{-1/2}).
\]
The King Cauchy identity for types $C_{n}\times C_{m}$ can be written:%
\begin{equation}
\prod_{i=1}^{n}\prod_{j=1}^{m}(y_{j}+y_{j}^{-1}+x_{i}+x_{i}^{-1}%
)=\sum_{\lambda\subseteq (m^{n})}(-1)^{\left\vert \lambda\right\vert }s_{\lambda
}^{C_{n}}(x)\,s_{\widehat{\lambda}}^{C_{m}}(y). \label{King}%
\end{equation}
By transforming $s_{\widehat{\lambda}}^{C_{m}}(y)$ according to (\ref{identity}),
one obtains the following Cauchy identity for types $C_{n}\times B_{m}$%
\begin{equation}
\prod_{j=1}^{m}(y_{j}^{1/2}+y_{j}^{-1/2})\prod_{i=1}^{n}\prod_{j=1}^{m}%
(y_{j}+y_{j}^{-1}+x_{i}+x_{i}^{-1})=\sum_{\lambda\subseteq (m^{n})}s_{\lambda
}^{C_{n}}(x)\,s_{\widehat{\lambda}+\omega_{m}^{B_{m}}}^{B_{m}}(y). \label{Dual}%
\end{equation}

\subsection{Type \texorpdfstring{$C_{n}\times B_{m}$}{Cn x Bm} Cauchy identity for the unequal-parameter Hall-Littlewood polynomials}

In order to equate more 1-d sums to deformations of weight multiplicities, we
need to use a two-parameter deformation of weight multiplicities of type
$B_{m}$, in $u$ (associated with the short roots $\varepsilon_{i}%
,i=1,\ldots,m$) and $q$ (associated with the long roots $\varepsilon_{i}%
\pm\varepsilon_{j},1\leq i<j\leq n$).\ Let us first start with a result
establish by Rains and Warnaar equating the unequal parameter Hall-Littlewood polynomials
to Koornwinder specializations (see Lemma 2.3 in \cite{RW}).

\begin{proposition}
\label{Prop-RW}For any partition $\lambda$ contained in the rectangle
$(n^{m})$, we have
\[
\prod_{j=1}^{m}(y_{j}^{1/2}+y_{j}^{-1/2})\,P_{\widehat{\lambda}}%
(y;u,0,0,0;0,q)=P_{\widehat{\lambda}+\omega_{m}}^{B_{m}}(y;u,q)
\]
and%
\[
P_{\widehat{\lambda}}(y;u,-1,0,0;0,q)=P_{\widehat{\lambda}}^{B_{m}}(y;u,q).
\]

\end{proposition}

One can observe that the second equality also follows from our specialization
Table \ref{table_spec} and the consideration exposed in \S \ \ref{subsec_HL}%
.\ Let $\mathrm{char}_{\leq n}^{B_{m},\mathrm{half}}(y)$ be the subspace of
the character ring of type $B_{m}$ with basis $\{s_{\widehat{\lambda}+\omega_{m}%
},\lambda\subseteq(m^{n})\}$.\ We can introduce a paring on $\mathrm{char}_{\leq
m}^{C_{n}}(x)\times\mathrm{char}_{\leq n}^{B_{m},\mathrm{half}}(y)$ such that
$\langle s_{\mu}^{C_{n}},s_{\widehat{\lambda}+\omega_{m}}^{B_{m}}\rangle
_{C_{n}\times B_{m}}=\delta_{\lambda,\mu}$.\ Let $\mathsf{\tilde{Q}}_{\mu
}^{C_{n}}(x,u,q)$ be the dual polynomial of the Hall-Littlewood polynomial
$P_{\widehat{\mu}+\omega_{m}}^{B_{m}}(y,u,q)$ for this pairing. We then have
\[
\mathsf{\tilde{Q}}_{\mu}^{C_{n}}(x;u,q)=\sum_{\lambda}K_{\widehat{\lambda}%
+\omega_{m},\widehat{\mu}+\omega_{m}}^{B_{m}}(u,q)\,s_{\lambda}^{C_{n}}(x)\,,
\]
and also
\begin{equation}
\prod_{j=1}^{m}(y_{j}^{1/2}+y_{j}^{-1/2})\prod_{i=1}^{n}\prod_{j=1}^{m}%
(y_{j}+y_{j}^{-1}+x_{i}+x_{i}^{-1})=\sum_{\mu\subseteq (m^{n})}\mathsf{\tilde{Q}%
}_{\mu}^{C_{n}}(x,u,q)\,P_{\widehat{\mu}+\omega_{m}}^{B_{m}}(y;u,q)\,, \label{KingHLB}%
\end{equation}
by using (\ref{Dual}). In particular when $q=u=0$, we recover exactly
(\ref{Dual}).
\subsection{The \texorpdfstring{$X=K$}{X=K} phenomenon in type \texorpdfstring{$A_{2n}^{(2)}$}{A{2n}(2)}}

Let first observe that we have also the following analogue of the dual Cauchy
formula (\ref{Cauchy_Koorn}) for the Koornwinder polynomials%
\[
\prod_{i=1}^{n}\prod_{j=1}^{m}(y_{j}+y_{j}^{-1}-x_{i}-x_{i}^{-1})=\sum
_{\mu\subseteq (m^{n})}(-1)^{\left\vert \mu\right\vert }P_{\mu}%
(x;a,b,c,d;q,t)\,P_{\widehat{\mu}}(y;a,b,c,d;t,q)\,,
\]
and therefore%
\[
\prod_{i=1}^{n}\prod_{j=1}^{m}(x_{i}+x_{i}^{-1}+y_{j}+y_{j}^{-1})=\sum
_{\mu\subseteq (m^{n})}(-1)^{\left\vert \mu\right\vert }P_{\mu}%
(-x;a,b,c,d;q,t)\,P_{\widehat{\mu}}(y;a,b,c,d;t,q).
\]
Let us multiply both sides by the character $s_{\omega_{m}}^{B_{m}}(y)$.
This gives%
\begin{multline*}
\prod_{j=1}^{m}(y_{j}^{1/2}+y_{j}^{-1/2})\prod_{i=1}^{n}\prod_{j=1}^{m}%
(x_{i}+x_{i}^{-1}+y_{j}+y_{j}^{-1})=\\
\sum_{\mu\subseteq (m^{n})}(-1)^{\left\vert \mu\right\vert }P_{\mu}%
(-x;a,b,c,d;q,t)\prod_{j=1}^{m}(y_{j}^{1/2}+y_{j}^{-1/2})\,P_{\widehat{\mu}%
}(y;a,b,c,d;t,q).
\end{multline*}
Now we will use the specialization $(-x;a,b,c,d;q,t)=(-x;u,0,0,0;q,0)$ and
obtain%
\begin{multline*}
\prod_{j=1}^{m}(y_{j}^{1/2}+y_{j}^{-1/2})\prod_{i=1}^{n}\prod_{j=1}^{m}%
(x_{i}+x_{i}^{-1}+y_{j}+y_{j}^{-1})  =\\
 \sum_{\mu\subseteq (m^{n})}(-1)^{\left\vert \mu\right\vert }P_{\mu
}(-x;u,0,0,0;q,0)\prod_{j=1}^{m}(y_{j}^{1/2}+y_{j}^{-1/2})\,P_{\widehat{\mu}%
}(y;u,0,0,0;0,q).
\end{multline*}
Then, set $u=-p,$ $q=p^{2}$ and consider the specialization
$(-x;a,b,c,d;q,t)=(-x;0,-p,0,0,0;p^{2},0)$ in the previous equality. It makes
appear the polynomials%
\[
P_{\mu}(-x;-p,0,0,0,0;p^{2},0)\,,
\]
which, up to sign flip $x\leftrightarrow-x$, are Demazure characters of type
$A_{2n}^{(2)}$ by the specialization Table \ref{table_spec} in which we have
to replace $q^{1/2}$ by $p$, that is Demazure characters of type $A_{2n}%
^{(2)}$ evaluated in $p^{2}$ instead of $q$. Since $\mathsf{\tilde{Q}}_{\mu
}^{C_{n}}(x;-p,p^{2})$ is the dual polynomial of the Hall-Littlewood
polynomial $P_{\widehat{\mu}+\omega_{m}}^{B_{m}}(y,-p,p^{2})$ for the pairing
$\langle\cdot,\cdot\rangle_{C_{n}\times B_{m}}$, and we have
\[
P_{\widehat{\mu}+\omega_{m}}^{B_{m}}(y;-p,p^{2})=\prod_{j=1}^{m}(y_{j}^{1/2}%
+y_{j}^{-1/2})\,P_{\widehat{\mu}}(y;u,0,0,0;0,q)
\]
by Proposition \ref{Prop-RW}, we can conclude by arguments similar to the
previous ones that
\[
(-1)^{\left\vert \mu\right\vert }P_{\mu}(-x;-p,0,0,0;p^{2},0)=\mathsf{\tilde
{Q}}_{\mu}^{C_{n}}(x;-p,p^{2})=\sum_{\lambda}K_{\widehat{\lambda}+\omega_{m}%
,\widehat{\mu}+\omega_{m}}^{B_{m}}(-p,p^{2})\,s_{\lambda}^{C_{n}}(x).
\]
Therefore, we have%
\[
P_{\mu}(-x;-p,0,0,0;p^{2},0)=\sum_{\lambda}(-1)^{\left\vert \mu\right\vert
}K_{\widehat{\lambda}+\omega_{m},\widehat{\mu}+\omega_{m}}^{B_{m}}(-p,p^{2}%
)\,s_{\lambda}^{C_{n}}(x).
\]
But since we have for any Weyl character of type $C_{n}$ the identity
$s_{\lambda}^{C_{n}}(-x)=(-1)^{\left\vert \mu\right\vert }s_{\lambda}^{C_{n}%
}(x)$, we can drop the signs in the set of variables $x$ and get%
\[
P_{\mu}(x;-p,0,0,0;p^{2},0)=\sum_{\lambda}(-1)^{\left\vert \lambda\right\vert
+\left\vert \mu\right\vert }K_{\widehat{\lambda}+\omega_{m},\widehat{\mu}+\omega_{m}%
}^{B_{m}}(-p,p^{2})\,s_{\lambda}^{C_{n}}(x).
\]
Then Theorem \ref{Th_Bogdan} tells us that $(-1)^{\left\vert \lambda
\right\vert +\left\vert \mu\right\vert }K_{\widehat{\lambda}+\omega_{m},\widehat{\mu
}+\omega_{m}}^{B_{m}}(-p,p^{2})$ is a 1-d sum of type $A_{2n}^{(2)}$ evaluated
in $p^{2}$ instead of $q$. At first glance, the signs seem problematic but in
fact they simplify as we will now explain.
Let us establish the lemma below
\begin{lemma}
    In the polynomials $K_{\widehat{\lambda}+\omega_{m},\widehat{\mu}+\omega_{m}}^{B_{m}
}(p,p^{2})$, all the powers $p^{k}$ which appear are even (resp. odd) when
$\left\vert \widehat{\lambda}\right\vert -\left\vert \widehat{\mu}\right\vert $ is
even (resp. odd).
\end{lemma}

\begin{proof}
    Observe first that for the $(p,p^{2})$-partition function $\mathcal{P}_{p,p^{2}}^{B_{n}}$, the polynomial $\mathcal{P}_{p,p^{2}}^{B_{n}}(\beta),\beta\in\mathbb{Z}^{n}$ has a partity equal to that of $\left\vert \beta\right\vert =\beta_{1}+\cdots+\beta_{n}$. The lemma follows because for any element $w$ in the Weyl group of type $B_m$, the parity of the integer $\left\vert w(\widehat{\lambda}+\omega_{m})-(\widehat{\mu}+\omega_{m})\right\vert$ is equal to that of $\left\vert \widehat{\lambda}-\widehat{\mu}\right\vert$. 
\end{proof}

By using the previous lemma, we obtain
\[
K_{\widehat{\lambda}+\omega_{m},\widehat{\mu}+\omega_{m}}^{B_{m}}(-p,p^{2}%
)=(-1)^{\left\vert \widehat{\lambda}\right\vert -\left\vert \widehat{\mu}\right\vert
}K_{\widehat{\lambda}+\omega_{m},\widehat{\mu}+\omega_{m}}^{B_{m}}(p,p^{2}).
\]
Now observe that%
\[
\left\vert \widehat{\lambda}\right\vert -\left\vert \widehat{\mu}\right\vert
=(nm-\left\vert \lambda\right\vert )-(nm-\left\vert \mu\right\vert
)=-\left\vert \lambda\right\vert +\left\vert \mu\right\vert =\left\vert
\lambda\right\vert +\left\vert \mu\right\vert \;\,\operatorname{mod}2\,,
\]
and we thus have%
\[
(-1)^{\left\vert \lambda\right\vert +\left\vert \mu\right\vert }%
K_{\widehat{\lambda}+\omega_{m},\widehat{\mu}+\omega_{m}}^{B_{m}}(-p,p^{2}%
)=K_{\widehat{\lambda}+\omega_{m},\widehat{\mu}+\omega_{m}}^{B_{m}}(p,p^{2}).
\]
This is eventually the modified Kostka-Foulkes polynomial $K_{\widehat{\lambda
}+\omega_{m},\widehat{\mu}+\omega_{m}}^{B_{m}}(p,p^{2})$, which is a 1-d sum of
type $A_{2n}^{(2)}$ evaluated in $p^{2}$.

Let $\lambda,\mu\subseteq(m^{n})$. We denote by $X_{\lambda,\mu}^{A_{2n}%
^{(2)}}(q)$ the one-dimensional sum associated with the dominant weight
$\lambda$ of the Kirillov-Reshetikhin crystal $B^{(\mu_{1}^{\prime},1)}%
\otimes\cdots\otimes B^{(\mu_{m}^{\prime},1)}$. Here $B^{(k,1)}$ denotes the
KR crystal of type $A_{2n-1}^{(2)}$ and column shape of height $k$. By
removing the $0$-arrows in $B^{(k,1)}$, we get a type $C_{n}$-crystal
isomorphic to the sum%
\[
B(\omega_{k})\oplus B(\omega_{k-1})\oplus\cdots\oplus B(\omega_{1})\oplus
B(0).
\]

\begin{theorem}\label{thm_A2n2}
For any pairs of partition $\lambda,\mu$ in the rectangle $(n^{m})$, we have%
\[
X_{\lambda,\mu}^{A_{2n}^{(2)}}(p^{2})=K_{\widehat{\lambda}+\omega_{m},\widehat{\mu
}+\omega_{m}}^{B_{m}}(p,p^{2}).
\]

\end{theorem}

\begin{remark}
    One may observe here (and similarly in the other analogous results that will be obtained in the forthcoming sections), that when $p=1$, we get a Howe-type duality between tensor product multiplicities of type $C_n$ and weight multiplicities of type $B_m$
\end{remark}

\begin{example}
Assume $n=m=3$ and put $\widehat{\lambda}+\omega_{m}=(5/2,3/2,1/2)$ $\widehat{\mu
}+\omega_{m}=(1/2,1/2,1/2)$. We get%
\[
K_{\widehat{\lambda}+\omega_{m},\widehat{\mu}+\omega_{m}}^{B_{m}}(p,p^{2}%
)=p^{13}+2q^{11}+3q^{9}+4q^{7}+3q^{5}+q^{3}.
\]
Then we have $\mu=(3,3,3)$ and $\lambda=(2,1,0)$. This polynomial is the
$A_{6}^{(2)}$ 1-d sum corresponding to the graded multiplicity of $\lambda$ in
the tensor product of $3$-KR column crystals of height $3$.
\end{example}

\begin{remark}
When $u=q$, we know that
\[
\prod_{j=1}^{m}(y_{j}^{1/2}+y_{j}^{-1/2})P_{\widehat{\lambda}}%
(y;q,0,0,0;0,q)=P_{\widehat{\lambda}+\omega_{m}}^{B_{m}}(y,q)
\]
is the Hall-Littlewood polynomial of type $B_{m}$ associated with the half-integer weight
$\widehat{\lambda}+\omega_{m}$. We can yet prove that we have
\[
P_{\mu}(x;q,0,0,0;q,0)=\sum_{\lambda}(-1)^{\left\vert \mu\right\vert
+\left\vert \lambda\right\vert }K_{\widehat{\lambda}+\omega_{m},\widehat{\mu}%
+\omega_{m}}^{B_{m}}(q)\,s_{\lambda}^{C_{n}}(x)\,,
\]
where the polynomials $K_{\widehat{\lambda}+\omega_{m},\widehat{\mu}+\omega_{m}%
}^{B_{m}}(q)$ are the ordinary one-parameter Lusztig $q$-analogues. But then,
the polynomial $P_{\lambda}(-x;q,0,0,0;q,0)$ is not a Demazure character of
type $A_{2n}^{(2)}$ since they come from the Koornwinder specialization%
\[
(a,b,c,d;q,t)=(t^{1/2},-t^{1/2},q^{1/2}t,-q^{1/2};q,0)\underset{t=0}%
{\rightarrow}(0,0,0,-q^{1/2},q,0).
\]
Hence we cannot claim that the 1-d sums of type $A_{2n}^{(2)}$ coincide with
the one-parameter KF polynomials of type $B_{m}$ and half weight (although
they do at $q=1$). Also the polynomials $(-1)^{\left\vert \mu\right\vert
+\left\vert \lambda\right\vert }K_{\widehat{\lambda}+\omega_{m},\widehat{\mu}%
+\omega_{m}}^{B_{m}}(q)$ have not nonnegative integer coefficients in general.
\end{remark}

\subsection{The \texorpdfstring{$X=K$}{X=K} phenomenon in type \texorpdfstring{$D_{n+1}^{(2)}$}{D{n+1}(2)}}

The ideas are quite similar as in the $A_{2n}^{(2)}$ case.\ We will see that
the 1-d sum then equates to two-parameter KF polynomials of type $B_{m}$ but
parametrized this time by partitions. There are nevertheless important
differences.\ We write the Mimachi formula~\eqref{Cauchy_Koorn} as follows:
\begin{equation}\label{mimachi-v2}
\prod_{i=1}^{n}\prod_{j=1}^{m}(y_{j}+y_{j}^{-1}-x_{i}-x_{i}^{-1}%
)=\sum_{\lambda\subseteq (m^{n})}(-1)^{\left\vert \lambda\right\vert }P_{\lambda
}(x;a,b,c,d;q,t)\,P_{\widehat{\lambda}}(y;a,b,c,d;t,q)\,,
\end{equation}
and we use this time the specialization $(x;a,b,c,d;q,t)=(x;u,-1,0,0;q,t)$, 
which by Proposition~\ref{Prop-RW} gives the two-parameter Hall-Littlewood polynomial of
type $B_{m}$%
\[
P_{\widehat{\lambda}}(y;u,q)=P_{\widehat{\lambda}}(y;u,-1,0,0;0,q).
\]
Here we do not need to multiply by the character $s_{\omega_{m}}^{B_{m}}(y).$
We get%
\begin{equation}
\prod_{i=1}^{n}\prod_{j=1}^{m}(y_{j}+y_{j}^{-1}-x_{i}-x_{i}^{-1})=\sum
_{\mu\subseteq (m^{n})}(-1)^{\left\vert \mu\right\vert }P_{\mu}%
(x;u,-1,0,0;q,0)\,P_{\widehat{\mu}}(y;u,q).\label{CauchyBnBm}%
\end{equation}
On the other hand, we have by Proposition 5\ in \cite{Golt}%
\[
\prod_{i=1}^{n}\prod_{j=1}^{m}(y_{j}+y_{j}^{-1}-x_{i}-x_{i}^{-1}%
)=\sum_{\lambda\subseteq (m^{n})}(-1)^{\left\vert \lambda\right\vert }s_{\lambda
}^{B_{n}}(x)\,s_{\widehat{\lambda}}^{B_{m}}(y).
\]
Let $\mathrm{char}_{\leq n}^{B_{m}}(y)$ (resp. Let $\mathrm{char}_{\leq
m}^{B_{m}}(x)$) be the subspace of the character ring of type $B_{m}$ (resp.
$B_{n}$) with basis $\{s_{\widehat{\lambda}}(y),\lambda\subseteq (m^{n})\}$ (resp.
$\{s_{\lambda}(x),\lambda\subseteq (m^{n})\}$). We can consider the pairing on
$\mathrm{char}_{\leq m}^{B_{n}}(x)\times\mathrm{char}_{\leq n}^{B_{m}}(y)$
such that $\langle s_{\mu}^{B_{n}},s_{\widehat{\lambda}}^{B_{m}}\rangle
_{B_{n}\times B_{m}}=(-1)^{\left\vert \lambda\right\vert }\delta_{\lambda,\mu
}$.\ Observe we use here a pairing with the two sides of type $B$.\ Let
$\mathsf{\tilde{Q}}_{\mu}^{B_{n}}(x;u,q)$ be the dual polynomial of the
Hall-Littlewood polynomial $P_{\widehat{\mu}}^{B_{m}}(y;u,q)$ for this pairing.
That is
\[
\langle\mathsf{\tilde{Q}}_{\mu}^{B_{n}}(x;u,q),P_{\widehat{\mu}}^{B_{m}%
}(y;u,q))=(-1)^{\left\vert \mu\right\vert }\delta_{\lambda,\mu}.
\]
We then have
\begin{equation}
\mathsf{\tilde{Q}}_{\mu}^{B_{n}}(x;u,q)=\sum_{\lambda}(-1)^{\left\vert
\lambda\right\vert +\left\vert \mu\right\vert }K_{\widehat{\lambda},\widehat{\mu}%
}^{B_{m}}(u,q)\,s_{\lambda}^{B_{n}}(x).\label{DecQ'Bn}
\end{equation}
Indeed, we can write
\[
\mathsf{\tilde{Q}}_{\mu}^{B_{n}}(x;u,q)=\sum_{\lambda}a_{\lambda,\mu
}s_{\lambda}^{B_{n}}(x)\,,
\]
where
\begin{equation*}
\begin{aligned}
a_{\lambda,\mu}&=(-1)^{\left\vert \lambda\right\vert }\langle\mathsf{\tilde{Q}%
}_{\mu}^{B_{n}}(x;u,q),s_{\widehat{\lambda}}^{B_{m}}(y)\rangle \\
&=(-1)^{\left\vert
\lambda\right\vert }\sum_{\widehat{\nu}}K_{\widehat{\lambda},\widehat{\nu}}^{B_{m}%
}(u,q)\langle\mathsf{\tilde{Q}}_{\mu}^{B_{n}}(x;u,q),P_{\widehat{\nu}}^{B_{m}%
}(y)\rangle \\
&=(-1)^{\left\vert \lambda\right\vert +\left\vert \mu\right\vert
}K_{\widehat{\lambda},\widehat{\mu}}^{B_{m}}(u,q).
\end{aligned}
\end{equation*}
We then get
\[
\prod_{i=1}^{n}\prod_{j=1}^{m}(y_{j}+y_{j}^{-1}-x_{i}-x_{i}^{-1})=\sum
_{\mu\subseteq(m^{n})}(-1)^{\left\vert \mu\right\vert }\mathsf{\tilde{Q}}_{\mu
}^{B_{n}}(x;u,q)\,P_{\widehat{\mu}}^{B_{m}}(y;u,q).
\]
Indeed
\begin{align*}
\prod_{i=1}^{n}\prod_{j=1}^{m}(y_{j}+y_{j}^{-1}-x_{i}-x_{i}^{-1}) &
=\sum_{\lambda\subseteq(m^{n})}(-1)^{\left\vert \lambda\right\vert }s_{\lambda
}^{B_{n}}(x)\,s_{\widehat{\lambda}}^{B_{m}}(y)=\\
\sum_{\mu}\left(  \sum_{\lambda}(-1)^{\left\vert \lambda\right\vert }%
K_{\widehat{\lambda},\widehat{\mu}}^{B_{m}}(u,q)s_{\lambda}^{B_{n}}(x)\right)
P_{\widehat{\mu}}^{B_{m}}(y;u,q) &  =\sum_{\mu}(-1)^{\left\vert \mu\right\vert
}\mathsf{\tilde{Q}}_{\mu}^{B_{n}}(x;u,q)\,P_{\widehat{\mu}}^{B_{m}}(y;u,q)\,,
\end{align*}
where we use (\ref{DecQ'Bn}). From (\ref{CauchyBnBm}), we so deduce the identity
\[
\mathsf{\tilde{Q}}_{\mu}^{B_{n}}(x;u,q)=P_{\mu}(x;u,-1,0,0;q,0).
\]
Now according to our specialization Table \ref{table_spec}, by setting $u=-p$
and $q=p^{2}$ we get from (\ref{DecQ'Bn}) that that $P_{\mu}(x,-p,-1,0,0;p^{2}%
,0)$ is a Demazure character of type $D_{n+1}^{(2)}$.

Recall that $X_{\lambda,\mu}^{D_{n+1}^{(2)}}(q)$ the one-dimensional sum of
type $D_{n+1}^{(2)}$ associated with the tensor product of columns defined by
$\mu$ (with $m$ columns) and the weight $\lambda$. Here the column KR crystals
$B^{(k,1)}$ of type $D_{n+1}^{(2)}$ have a classical structure (obtained by
removing the $0$-arrows) of type $B_{n}$ isomorphic to%
\[
B(\omega_{k})\oplus B(\omega_{k-1})\oplus\cdots\oplus B(\omega_{1})\oplus
B(0).
\]
We have proved the following theorem.

\begin{theorem}\label{thm_Dn+12}
For any pair of partitions $\lambda,\mu$ in the rectangle $(n^{m})$, we have
\[
X_{\lambda,\mu}^{D_{n+1}^{(2)}}(q)=(-1)^{\left\vert \lambda\right\vert
+\left\vert \mu\right\vert }K_{\widehat{\lambda},\widehat{\mu}}^{B_{m}}(-p,p^{2}%
)=(-1)^{\left\vert \widehat{\lambda}\right\vert +\left\vert \widehat{\mu}\right\vert
}K_{\widehat{\lambda},\widehat{\mu}}^{B_{m}}(-p,p^{2})\,.
\]
In particular. the 1-d sums of type $D_{n+1}^{(2)}$ are signed KF
polynomials for the unequal parameters $-p$ and $p^{2}$.
\end{theorem}

\begin{remark}
We cannot simplify the signs appearing in the expressions 
$(-1)^{\left\vert\widehat{\lambda}\right\vert +\left\vert \widehat{\mu}\right\vert }K_{\widehat{\lambda},\widehat{\mu}}^{B_{m}}(-p,p^{2})$ to make appear the polynomials $K_{\widehat{\lambda},\widehat
{\mu}}^{B_{m}}(p,p^{2})$.\ This is coherent with the fact that these last
polynomials do not have nonnegative coefficients in general for $\widehat{\lambda
},\widehat{\mu}$ two partitions (this is nevertheless the case with half-integer
weights as explained previously).
\end{remark}

\section{The \texorpdfstring{$X=K$}{X=K} phenomenon in type \texorpdfstring{$A_{2n}^{(2,\dagger)}$}{A{2n}(2)} and type \texorpdfstring{$D_{m}$}{Dm}
Kostka-Foulkes polynomials for half-integer weights}
\label{sec_type_D_halfint}

To connect the type $D_{m}$ Lusztig $q$-analogues of type $D_{m}$ and half
integer weights with some 1-d sum, the idea is to consider the previous identity
\[
\mathsf{\tilde{Q}}_{\mu}^{B_{n}}(x;u,q)=P_{\mu}(x;u,-1,0,0;q,0)\,,
\]
and then specialize $u$ to $0$. We have
\[
K_{\widehat{\lambda},\widehat{\mu}}^{B_{m}}(u,q)=\sum_{w\in W_{B_{m}}}(-1)^{\ell
_{B_{m}}(w)}\mathcal{P}_{u,q}^{B_{m}}(w(\widehat{\lambda}+\rho_{B_{m}})-(\widehat{\mu
}+\rho_{B_{m}}))\,,
\]
where $\mathcal{P}_{u,q}$ is the $(u,q)$-Kostant partition function defined by%
\[
\prod_{i=1}^{m}(1-ux_{i})^{-1}\prod_{1\leq i<j\leq m}(1-qx_{i}x_{j}%
)=\sum_{\beta\in\mathbb{Z}^{m}}\mathcal{P}_{u,q}(\beta)x^{\beta}.
\]
Now when $u=0$, $\mathcal{P}_{0,q}=\mathcal{P}_{q}^{D_{m}}$ is the $q$-Kostant
partition function of type $D_{m}$. Moreover we have $W_{B_{m}}=W_{D_{m}}%
{\textstyle\bigsqcup}
W_{D_{m}}s_{\varepsilon_{m}}$ where $\iota=s_{\varepsilon_{m}}$ acts on
$\mathbb{Z}^{m}$ by changing the sign of the last coordinate.\ Also for any
$w$ in $W_{D_{m}},$ one has $(-1)^{\ell_{B_{m}}(w)}=(-1)^{\ell_{D_{m}}(w)}%
$.\ Finally $\rho_{B_{m}}=\rho_{D_{m}}+(1/2)^{m}$. This thus gives%
\[
K_{\widehat{\lambda},\widehat{\mu}}^{B_{m}}(0,q)=K_{\widehat{\lambda}+(1/2)^{m},\widehat{\mu
}+(1/2)^{m}}^{D_{m}}(q)-K_{\iota(\widehat{\lambda}+(1/2)^{m}),\widehat{\mu}+(1/2)^{m}%
}^{D_{m}}(q)
\]
where $\iota$ is the involution of the weight lattice induced by the type
$D_{m}$ Dynkin diagram automorphism permuting the nodes $m$ and $m-1$ (it
changes the sign of the last coordinates of the weights).\ Alternatively, we
also have
\[
K_{\widehat{\lambda},\widehat{\mu}}^{B_{m}}(0,q)=K_{\widehat{\lambda}+(1/2)^{m},\widehat{\mu
}+(1/2)^{m}}^{D_{m}}(q)-K_{\widehat{\lambda}+(1/2)^{m},\iota(\widehat{\mu}+(1/2)^{m}%
)}^{D_{m}}(q)
\]
because $\mathcal{P}^{D_{m}}(\beta)=\mathcal{P}^{D_{m}}(\iota(\beta))$.\ 

In fact this difference simplifies.\ To see this, recall that $\beta
\in\mathbb{Z}^{m}$ belongs to the set $Q_{+}^{D_{m}}$ of nonnegative
combinations of positive roots of type $D_{m}$ if and only if $\beta
_{1}+\cdots+\beta_{i}\geq0$ for any $i=1,\ldots,m$ and $\left\vert
\beta\right\vert =\beta_{1}+\cdots+\beta_{m}$ is even.\ Assume $K_{\widehat
{\lambda}+(1/2)^{m},\widehat{\mu}+(1/2)^{m}}^{D_{m}}(q)\neq0$. Then $\widehat{\lambda
}+(1/2)^{m}-(\widehat{\mu}+(1/2)^{m})$ belongs to $Q_{+}^{D_{m}}$ and therefore
$\left\vert \widehat{\lambda}\right\vert -\left\vert \widehat{\mu}\right\vert $ is
even.\ But now $\widehat{\lambda}+(1/2)^{m}-\iota(\widehat{\mu}+(1/2)^{m})=\left\vert
\widehat{\lambda}\right\vert -\left\vert \widehat{\mu}\right\vert +2\widehat{\mu}_{m}+1$ is odd
therefore $K_{\widehat{\lambda}+(1/2)^{m},\iota(\widehat{\mu}+(1/2)^{m})}^{D_{m}}(q)=0$.\ Similarly, when $K_{\widehat{\lambda}+(1/2)^{m},\iota(\widehat{\mu}%
+(1/2)^{m})}^{D_{m}}(q)\neq0$, then $K_{\widehat{\lambda}+(1/2)^{m},\widehat{\mu
}+(1/2)^{m}}^{D_{m}}(q)=0$.\ Finally, we have
\[
K_{\widehat{\lambda},\widehat{\mu}}^{B_{m}}(0,q)=\left\{
\begin{array}
[c]{c}%
K_{\widehat{\lambda}+(1/2)^{m},\widehat{\mu}+(1/2)^{m}}^{D_{m}}(q)\text{ when
}\left\vert \widehat{\lambda}\right\vert -\left\vert \widehat{\mu}\right\vert \text{
is even,}\\
-K_{\widehat{\lambda}+(1/2)^{m},\iota(\widehat{\mu}+(1/2)^{m})}^{D_{m}}(q)\text{when
}\left\vert \widehat{\lambda}\right\vert -\left\vert \widehat{\mu}\right\vert \text{
is odd,}%
\end{array}
\right.
\]
so that%
\[
(-1)^{\left\vert \lambda\right\vert +\left\vert \mu\right\vert }%
K_{\widehat{\lambda},\widehat{\mu}}^{B_{m}}(0,q)=\left\{
\begin{array}
[c]{c}%
K_{\widehat{\lambda}+(1/2)^{m},\widehat{\mu}+(1/2)^{m}}^{D_{m}}(q)\text{ when
}\left\vert \widehat{\lambda}\right\vert -\left\vert \widehat{\mu}\right\vert \text{
is even,}\\
K_{\widehat{\lambda}+(1/2)^{m},\iota(\widehat{\mu}+(1/2)^{m})}^{D_{m}}(q)\text{when
}\left\vert \widehat{\lambda}\right\vert -\left\vert \widehat{\mu}\right\vert \text{
is odd.}%
\end{array}
\right.
\]

With our specialization $u=0$, we moreover get%
\[
\mathsf{\tilde{Q}}_{\mu}^{B_{n}}(x;0,q)=P_{\mu}(x,0,-1,0,0;0,0)=\sum_{\lambda
}(-1)^{\left\vert \lambda\right\vert +\left\vert \mu\right\vert }%
K_{\widehat{\lambda},\widehat{\mu}}^{B_{m}}(0,q)\,s_{\lambda}^{B_{n}}(x).
\]
According to our specialization table \ref{table_spec}, the polynomial
$P_{\mu}(x;0,-1,0,0;0,0)$ is an affine Demazure character of type
$A_{2n}^{(2,\dagger)}$.\ Observe here the difference compared to the case of half-integer weights in type $B_{m}$.

Denote by $X_{\lambda,\mu}^{A_{2n}^{(2,\dagger)}}(q)$ the one-dimensional sum
of type $A_{2n}^{(2,\dagger)}$ associated with the tensor product of columns
defined by $\mu$ (with $m$ columns) and the weight $\lambda$. Here the column
KR crystals $B^{(k,1)}$ of type $A_{2n}^{(2,\dagger)}$ have a classical
structure (obtained by removing the $0$-arrows) of type $B_{n}$ isomorphic to
the connected crystal $B(\omega_{k})$. We have proved the following theorem

\begin{theorem}\label{thm_A2n2dag}
For any pair of partitions $\lambda,\mu$ in the rectangle $(n^{m})$, we have
\[
X_{\lambda,\mu}^{A_{2n}^{(2,\dagger)}}(q)=\left\{
\begin{array}
[c]{c}%
K_{\widehat{\lambda}+(1/2)^{m},\widehat{\mu}+(1/2)^{m}}^{D_{m}}(q)\text{ when
}\left\vert \widehat{\lambda}\right\vert -\left\vert \widehat{\mu}\right\vert \text{
is even,}\\
K_{\widehat{\lambda}+(1/2)^{m},\iota(\widehat{\mu}+(1/2)^{m})}^{D_{m}}(q)\text{ when
}\left\vert \widehat{\lambda}\right\vert -\left\vert \widehat{\mu}\right\vert \text{
is odd.}%
\end{array}
\right.
\]
\end{theorem}

\section{Untwisted cases \texorpdfstring{$B_{n}^{(1)},\;B_{n}^{(1,\dagger)},\;C_{n}^{(1)},$}{} and \texorpdfstring{$D_{n}^{(1)}$}{D{n}(1)}}
\label{sec_nontwisted}

\subsection{General considerations about the untwisted cases}

Let us study whether we can obtain 1-d sums of untwisted types by similar
techniques as KF polynomials. First of all, we need to make appear
Hall-Littlewood polynomials thanks to specializations in the Koornwinder polynomials
$P_{\widehat{\mu}}(y,a,b,c,d;t,q)$ (keep in mind the flip of $(q,t)$ into
$(t,q)$). This can be done in several ways from Macdonald specializations at
$t=0$ but which will eventually produce the same Hall-Littlewood polynomial at the end. For
example Hall-Littlewood polynomials of type $C_{m}$ can be obtained from any affine root
system whose classical finite subroot system (obtained by removing the zero
node) is of type $C_{m}$, hence $C_{m}^{(1)}$, $A_{2m-1}^{(2)}$, and $A_{2n}%
^{(2)}$. We get the following table, where we use the parameter $u$ (a priori different from $q$) related to the orbits of $\varepsilon_{n}$ or
$2\varepsilon_{n}$.%
\begin{equation}%
\begin{tabular}[c]{ll}\hline
Type & $(y;a,b,c,d;t,q)$\\
\hline
$B_{m}$ integer weights & $(y;u,-1,0,0;0,q)$\\
$B_{m}$ half-integer weights & $(y;u,0,0,0;0,q)$\\
$C_{m}$ & $(y;u^{1/2},-u^{1/2},0;0,q)$\\
$D_{m}$ & $(y;1,-1,0,0;0,q)$\\\hline
\end{tabular}
\ \ \ \ \ \ \ \label{Table HL}%
\end{equation}
Now, if we want to make appear
1-d sums of untwisted affine types from Koornwinder polynomials beyond
type $A_{n-1}^{(1)}$, the specialization in the parameters $(x,a,b,c,d;q,t)$
should be done according to the table below.
\begin{subequations}
\label{0}%
\begin{equation}%
\begin{tabular}[c]{ll}\hline
Type & $(T_{N}^{(a)},a,b,c,d;q,t)$\\\hline
$B_{n}^{(1)}$ & $(x;0,-1,q^{1/2},-q^{1/2};q,0)$\\
$B_{n}^{(1,\dagger)}$ & $(x;0,1,-1,-q^{1/2};q,0)$\\
$C_{n}^{(1)}$ & $(x;0,0,0,0;q,0)$\\
$D_{n}^{(1)}$ & $(x;-1,1,q^{1/2},-q^{1/2};q,0)$\\\hline
\end{tabular}
\ \ \ \ \ \ \ \label{Table1dsumsNT}%
\end{equation}
One immediately sees that the in the first (HL specialization) table, at most two
parameters $(a,b,c,d)$ are non-zero.\ Therefore, there is no chance that these
specializations can make appear 1-d sums of type $B_{n}^{(1)}$, $B_{n}%
^{(1,\dagger)}$, or $D_{n}^{(1)}$, where we need at least three nonzero
parameters. In contrast, we can make appear 1-d sums of type $C_{n}^{(1)}$
from type $C_{m}$ two parameter HL specialization where we put $u=0$. This
will be studied in the following paragraph.
\end{subequations}

\subsection{Type \texorpdfstring{$C_{n}^{(1)}$}{C{n}(1)} 1-d sums}

We also start from~\eqref{mimachi-v2}, 
%\[
%\prod_{i=1}^{n}\prod_{j=1}^{m}(y_{j}+y_{j}^{-1}-x_{i}-x_{i}^{-1})=\sum_{\mu\subseteq(m^{n})}(-1)^{\left\vert \widehat{\mu}\right\vert }P_{\mu}(x;a,b,c,d;q,t)\,P_{\widehat{\mu}}(y;a,b,c,d;t,q)
%\]
and we use the specialization $(x;a,b,c,d;q,t)=(x;0,0,0,0;q,0)$, which gives
the two-parameter Hall-Littlewood polynomial of type $C_{m}$%
\[
P_{\widehat{\mu}}^{C_{m}}(y;0,q)=P_{\widehat{\mu}}(y;0,0,0,0;0,q).
\]
We get%
\[
\prod_{i=1}^{n}\prod_{j=1}^{m}(y_{j}+y_{j}^{-1}-x_{i}-x_{i}^{-1})=\sum
_{\mu\subseteq(m^{n})}(-1)^{\left\vert \mu\right\vert }P_{\mu}%
(x;0,0,0,0;q,0)\,P_{\widehat{\mu}}(y;0,q)\,,
\]
but since the $P_{\mu}(x,0,0,0,0;q,0)$'s are in the character ring of type
$C_{n}$ (because this is a Demazure character of type $C_{n}^{(1)}$), we
obtain
\[
\prod_{i=1}^{n}\prod_{j=1}^{m}(y_{j}+y_{j}^{-1}+x_{i}+x_{i}^{-1})=\sum
_{\mu\subseteq(m^{n})}P_{\mu}(x;0,0,0,0;q,0)\,P_{\widehat{\mu}}(y;0,q)\,,
\]
and%
\[
\prod_{i=1}^{n}\prod_{j=1}^{m}(y_{j}+y_{j}^{-1}+x_{i}+x_{i}^{-1}%
)=\sum_{\lambda\subseteq(m^{n})}s_{\lambda}^{C_{n}}(x)\,s_{\widehat{\lambda}}^{C_{m}%
}(y).
\]
We can consider yet the pairing on $\mathrm{char}_{\leq m}^{C_{n}}%
(x)\times\mathrm{char}_{\leq n}^{C_{m}}(y)$ such that $\langle s_{\mu}^{C_{n}%
},s_{\widehat{\lambda}}^{C_{m}})_{C_{n}\times C_{m}}=\delta_{\lambda,\mu}$.\ Let
$\widehat{\mathsf{Q}}_{\mu}^{C_{n}}(x;0,q)$ be the dual polynomial of the
Hall-Littlewood polynomial $P_{\widehat{\mu}}^{C_{m}}(y;0,q)$ for this pairing.
That is
\[
\langle\widehat{\mathsf{Q}}_{\mu}^{C_{n}}(x;0,q),P_{\widehat{\mu}}^{C_{m}%
}(y;0,q)\rangle_{C_{n}\times C_{m}}=\delta_{\lambda,\mu}.
\]
This should not be confused with the polynomial $\mathsf{Q}_{\mu}^{C_{n}}(x;q)$
used in Section \ref{Sec_A_(2n-1)^(2)} which is the dual of the ordinary
(one-parameter) Hall-Littlewood polynomial of type $C_{m}$. We then have
\[
\widehat{\mathsf{Q}}_{\mu}^{C_{n}}(x;0,q)=\sum_{\lambda}K_{\widehat{\lambda}%
,\widehat{\mu}}^{C_{m}}(0,q)\,s_{\lambda}^{C_{n}}(x)\,,
\]
and we get
\[
\prod_{i=1}^{n}\prod_{j=1}^{m}(y_{j}+y_{j}^{-1}+x_{i}+x_{i}^{-1})=\sum
_{\mu\subseteq(m^{n})}\widehat{\mathsf{Q}}_{\mu}^{C_{n}}(x;0,q)\,P_{\widehat{\mu}%
}^{C_{m}}(y;0,q).
\]
We deduce the identity%
\[
\widehat{\mathsf{Q}}_{\mu}^{C_{n}}(x;0,q)=P_{\mu}(x;0,0,0,0;q,0).
\]
Therefore the one-dimensional sums of type $C_{n}^{(1)}$ associated with the
tensor product of columns defined by $\mu$ (with $m$ columns) coincide with
the signed KF polynomials $K_{\widehat{\lambda},\widehat{\mu}}^{C_{m}}(0,q)$ of type
$C_{m}$. Let us now examine more precisely what are these polynomials
$K_{\widehat{\lambda},\widehat{\mu}}^{C_{m}}(0,q)$.\ We have%
\[
K_{\widehat{\lambda},\widehat{\mu}}^{C_{m}}(0,q)=\sum_{w\in W_{C_{m}}}\varepsilon
(w)\mathcal{P}_{0,q}^{C_{m}}(w(\widehat{\lambda}+\rho_{C_{m}})-(\widehat{\mu}%
+\rho_{C_{m}}))\,,
\]
and the Kostant partition function $\mathcal{P}_{0,q}^{C_{m}}$ is nothing but
the Kostant partition $\mathcal{P}_{0,q}^{D_{m}}$ function for type $D_{m}$.
We also have $\rho_{C_{m}}=\rho_{D_{m}}+(1)^{m}$ and $W_{C_{m}}=W_{D_{m}}%
{\textstyle\bigsqcup}
W_{D_{m}}s_{\varepsilon_{m}}$.\ Therefore, we get%
\[
K_{\widehat{\lambda},\widehat{\mu}}^{C_{m}}(0,q)=K_{\widehat{\lambda}+(1)^{m},\widehat{\mu
}+(1)^{m}}^{D_{m}}(q)-K_{\widehat{\lambda}+(1)^{m},\iota(\widehat{\mu}+(1)^{m}%
)}^{D_{m}}(q).
\]
Observe that contrary to the previous case of the polynomials $K_{\widehat
{\lambda},\widehat{\mu}}^{B_{m}}(0,q)$, we cannot conclude by saying that
$K_{\widehat{\lambda}+(1)^{m},\widehat{\mu}+(1)^{m}}^{D_{m}}(q)$ and $K_{\widehat{\lambda
}+(1)^{m},\iota(\widehat{\mu}+(1)^{m})}^{D_{m}}(q)$ cannot be simultaneously
nonzero polynomials.

Denote by $X_{\lambda,\mu}^{C_{n}^{(1)}}(q)$ the one-dimensional sum of type
$C_{n}^{(1)}$ associated with the tensor product of columns defined by $\mu$
(with $m$ columns) and the weight $\lambda$. Here the column KR crystals
$B^{(k,1)}$ of type $C_{n}^{(1)}$ have a classical structure (obtained by
removing the $0$-arrows) of type $C_{n}$ isomorphic to $B^{C_{n}}(\omega_{k}%
)$. We have proved the following theorem.

\begin{theorem}\label{thm_Cn1}
For any pair of partitions $\lambda,\mu$ in the rectangle $(n^{m})$, we have
\[
X_{\lambda,\mu}^{C_{n}^{(1)}}(q)=K_{\widehat{\lambda},\widehat{\mu}}^{C_{m}%
}(0,q)=K_{\widehat{\lambda}+(1)^{m},\widehat{\mu}+(1)^{m}}^{D_{m}}(q)-K_{\widehat{\lambda
}+(1)^{m},\iota(\widehat{\mu}+(1)^{m})}^{D_{m}}(q).
\]

\end{theorem}

\section{A worked example}

\label{SecExample}
We illustrate \Cref{thm_An-11}, \Cref{thm_A2n-12}, \Cref{thm_A2n-12dag}, \Cref{thm_Dn+12}, \Cref{thm_A2n2}, \Cref{thm_A2n2dag}, \Cref{thm_Cn1} on an example.
Take $n=4$, $m = 4$, and 
$$\mu = (4,3,2,0) = \gyoung(<><><><>,<><><>,<><>) \text{\quad so that \quad }
\widehat{\mu} = (3,2,1,1)=\gyoung(<><><>,<><>,<>,<>).$$
Choose 
$$\la=(1,0,0,0)=\gyoung(<>) \text{\quad so that \quad }\widehat{\la} = (4,4,4,3)=\gyoung(<><><><>,<><><><>,<><><><>,<><><>).$$

\medskip

The one-dimensional sums are computed using $B^{(3,1)}\otimes B^{(3,1)}\otimes B^{(2,1)}\otimes B^{(1,1)}$, 
the tensor product of column Kirillov-Reshetikhin crystals of shape
$\mu'=(3,3,2,1)$, that is
$$
{
%\scriptsize
\gyoung(<>,<>,<>)
\otimes
\gyoung(<>,<>,<>)
\otimes
\gyoung(<>,<>)
\otimes
\gyoung(<>)
}.
$$

\subsection{Type \texorpdfstring{$A_{n-1}^{(1)}$}{A{n-1}(1)}}

This is the classical setting of \cite{NY1997}, which does not require $\widehat\mu$ nor $\widehat\la$,
but $\mu'$ and $\la'$ instead.

\medskip

Let us first compute the one-dimensional sum $X_{\la,\mu}(q)$, which can be done for instance in Sage.
We get the following three highest weight vertices with corresponding energy function:
$$
\begin{array}{ll}
\hline
\text{Highest weight vertex} & \text{Energy}  
\\
\hline
{
\scriptsize
\gyoung(<1>,<3>,<4>)\otimes\gyoung(<1>,<2>,<4>)\otimes\gyoung(<2>,<3>)\otimes\gyoung(<1>)
}

&
4
\\
{\scriptsize\gyoung(<1>,<3>,<4>)\otimes\gyoung(<2>,<3>,<4>)\otimes\gyoung(<1>,<2>)\otimes\gyoung(<1>)}
&
2
\\
{\scriptsize\gyoung(<2>,<3>,<4>)\otimes\gyoung(<1>,<3>,<4>)\otimes\gyoung(<1>,<2>)\otimes\gyoung(<1>)}
&
3
\vspace{5pt}
\\
\hline
\end{array}
$$
which yields
$$
X_{\lambda ,\mu }^{A_{3}^{(1)}}(q)=q^{4}+q^{3}+q^{2}.
$$
In fact, in \Cref{thm_An-11}, in order to compute the corresponding Kostka-Foulkes polynomial,
we must have $|\la|=|\mu|$, which is not the case here,
but we can replace $\la$ by $\dot\la=(3,2,2,2)$ since $\dot\la-\la = 2.(1,1,1,1)$
so $\la$ and $\dot\la$ coincide as $\mathfrak{sl}_4$-weights.
One check that a direct computation of the Kostka-Foulkes polynomial gives
$$K_{{\dot\lambda}',\mu'}^{A_{3}}(q) = q^{4}+q^{3}+q^{2}.$$

\subsection{Type \texorpdfstring{$C_{n}^{(1)}$}{C{n}(1)}}

We now first illustrate \Cref{thm_Cn1} (for whom the computation is a bit lighter).
The energy function is given by the following table.

$$
\begin{array}{ll}
\hline
\text{Highest weight vertex} & \text{Energy}  
\\
\hline

{\scriptsize
\gyoung(<\overline{4}>,<\overline{3}>,<\overline{2}>)\otimes \gyoung(<2>,<3>,<4>)\otimes \gyoung(<2>,<\overline{2}>)\otimes \gyoung(<1>)
}
& 8
\\

{\scriptsize
\gyoung(<\overline{4}>,<\overline{3}>,<\overline{2}>)\otimes \gyoung(<3>,<4>,<\overline{1}>)\otimes \gyoung(<1>,<2>)\otimes \gyoung(<1>)
}
& 6
\\

{\scriptsize
\gyoung(<\overline{3}>,<\overline{2}>,<\overline{1}>)\otimes \gyoung(<1>,<2>,<3>)\otimes \gyoung(<2>,<\overline{2}>)\otimes \gyoung(<1>)
}
& 6
\\

{\scriptsize
\gyoung(<\overline{3}>,<\overline{2}>,<\overline{1}>)\otimes \gyoung(<1>,<4>,<\overline{4}>)\otimes \gyoung(<2>,<3>)\otimes \gyoung(<1>)
}
& 8
\\

{\scriptsize
\gyoung(<\overline{3}>,<\overline{2}>,<\overline{1}>)\otimes \gyoung(<3>,<4>,<\overline{4}>)\otimes \gyoung(<1>,<2>)\otimes \gyoung(<1>)
}
& 7
\\

{\scriptsize
\gyoung(<\overline{3}>,<\overline{2}>,<\overline{1}>)\otimes \gyoung(<2>,<3>,<\overline{2}>)\otimes \gyoung(<1>,<2>)\otimes \gyoung(<1>)
}
& 4
\\

{\scriptsize
\gyoung(<3>,<\overline{3}>,<\overline{2}>)\otimes \gyoung(<4>,<\overline{4}>,<\overline{3}>)\otimes \gyoung(<2>,<3>)\otimes \gyoung(<1>)
}
& 9
\\

{\scriptsize
\gyoung(<3>,<\overline{3}>,<\overline{2}>)\otimes \gyoung(<3>,<\overline{3}>,<\overline{1}>)\otimes \gyoung(<1>,<2>)\otimes \gyoung(<1>)
}
& 5
\\

{\scriptsize
\gyoung(<3>,<\overline{3}>,<\overline{2}>)\otimes \gyoung(<2>,<3>,<\overline{3}>)\otimes \gyoung(<2>,<\overline{2}>)\otimes \gyoung(<1>)
}
& 7
\vspace{5pt}
\\

\hline
\end{array}
$$
So we find
$$
X_{\lambda,\mu }^{C_{4}^{(1)}}(q)=q^{9}+2q^{8}+2q^{7}+2q^{6}+q^{5}+q^{4}=K_{\widehat{\lambda},\widehat{\mu}}^{C_{4}}(0,q^{2})\,,$$
where the Kostka-Foulkes polynomials are computed independently.

\subsection{Type \texorpdfstring{$A_{2n-1}^{(2)}$}{A{2n-1}(2)}}

We illustrate \Cref{thm_A2n-12} by
a similar computation, which gives
\begin{equation*}
X_{\lambda ,\mu
}^{A_{7}^{(2)}}(q)=q^{14}+q^{13}+2q^{12}+3q^{11}+4q^{10}+6q^{9}+7q^{8}+5q^{7}+4q^{6}+2q^{5}=K_{\widehat{\lambda},\widehat{\mu}}^{C_{4}}(q).
\end{equation*}

\subsection{Type \texorpdfstring{$A_{2n-1}^{(2,\dagger)}$}{A{2n-1}(2)}}

In type $A_{7}^{(2,\dagger )}$, we find the one-dimensional sum
\begin{equation*}
X_{\lambda ,\mu }^{D_{5}^{(2,\dagger
)}}(q)=q^{10}+q^{9}+2q^{8}+2q^{7}+2q^{6}+q^{5}+q^{4}=K_{\widehat{\lambda},\widehat{%
\mu}}^{D_{4}}(q).
\end{equation*}

\subsection{Type \texorpdfstring{$D_{n+1}^{(2)}$}{D{n+1}(2)}}

\begin{equation*}
\begin{aligned}
X_{\lambda ,\mu
}^{D_{5}^{(2)}}(q)&=q^{28}+2q^{26}+4q^{24}+8q^{22}+q^{21}+13q^{20}+q^{19}+19q^{18}+q^{17}+\\
&\qquad+24q^{16}+q^{15}+24q^{14}+19q^{12}+10q^{10}+3q^{8}\\
&=K_{\widehat{\lambda},\widehat{\mu}}^{B_{4}}(-q,q^{2}).
\end{aligned}
\end{equation*}

\subsection{Type \texorpdfstring{$A_{2n}^{(2)}$}{A{2n}(2)}}

In type $A_{8}^{(2)}$, we find the one-dimensional sum 
\begin{equation*}
\begin{aligned}
X_{\lambda ,\mu
}^{A_{8}^{(2)}}(q)&=q^{28}+2q^{26}+4q^{24}+8q^{22}+12q^{20}+19q^{18}+24q^{16}+24q^{14}+19q^{12}+10q^{10}+3q^{8}\\
&=K_{%
\widehat{\lambda}+(1/2)^{4},\widehat{\mu}+(1/2)^{4}}^{B_{4}}(q,q^{2}).
\end{aligned}
\end{equation*}

\subsection{Type \texorpdfstring{$A_{2n}^{(2,\dagger)}$}{A{2n}(2)}}

In type $A_{8}^{(2,\dagger )}$, we find the one-dimensional sum%
\begin{equation*}
X_{\lambda ,\mu }^{A_{8}^{(2,\dagger
)}}(q)=q^{10}+q^{9}+2q^{8}+2q^{7}+2q^{6}+q^{5}+q^{4}=K_{\widehat{\lambda}+(1/2)^{4},\widehat{\mu}+(1/2)^{4}}^{D_{4}}(q)
\end{equation*}

\section{Future works}\label{sec:final}

\begin{enumerate}
\item As mentioned briefly in the Introduction and explained in Section \ref{sec_nontwisted}, it is not possible to equate the one-dimensional sums associated with a tensor product of KR-crystals of type $B_{n}^{(1)}$ and $D_{n}^{(1)}$ with a generalized Kostka-Foulkes polynomials. We nevertheless think there are relevant extensions of the notion of Kostka-Foulkes polynomials (defined similarly from
alternating sums of suitable $q$-Kostant type partition functions) giving these
missing identifications.

\item The equalities illustrated in Table \ref{table} can be specialized at $q=1$ and
then give various Howe-type dualities. An interesting problem concerns the
generalization of the combinatorial Howe duality \cite{GGL} obtained in type $C_{n}$ which permits to get a charge statistic on King tableaux.\ 
More precisely, 
it would be interesting to have a combinatorial proof of the various Howe-type
dualities coming from the $q=1$ specialization of our results. As explained
in Remark~\ref{RemGGL}, transferring the energy statistic through this correspondence would
give a charge statistic on relevant combinatorial objects. For example, in type
$C_{n}^{(1)}$ the KR-column crystals are parametrized by the so-called
admissible columns. The duality described in Remark~\ref{RemGGL}, once restricted to the highest weight tensor products of such columns, gives a subset of King tableaux with a simple
combinatorial description, hence a combinatorial description of the
generalized Kostka-Foulkes polynomials $K_{\widehat{\lambda},\widehat{\mu}}^{C_{m}}(0,q)$.

\item Besides the study of the combinatorics mentioned above, we plan to continue developing the combinatorics of the quantum alcove model in~\cite{LNSSS,Lenart,LeSc} in the direction of the Kostka-Foulkes polynomials and the energy function, as suggested in the Introduction.
\end{enumerate}

% \bibliographystyle{plain}
% \bibliography{biblio}

\end{document}